\documentclass[hidelinks,onefignum,onetabnum]{siamart220329}

\usepackage{xparse}
\usepackage{amsfonts}
\usepackage{amssymb}
\usepackage{bbm}
\usepackage[Export]{adjustbox} 
\usepackage{nicematrix}
\usepackage{bm}

\usepackage{mathtools}
\usepackage{enumitem}

\usepackage{accents}
\newcommand{\ubar}[1]{\underaccent{\bar}{#1}}

\NewDocumentCommand{\mapr}{D(){{\alpha,\beta}}}{R_{#1}}

\newcommand{\leftrarrows}{\mathrel{\raise.75ex\hbox{\oalign{%
  $\scriptstyle\leftarrow$\cr
  \vrule width0pt height.5ex$\hfil\scriptstyle\relbar$\cr}}}}
\newcommand{\lrightarrows}{\mathrel{\raise.75ex\hbox{\oalign{%
  $\scriptstyle\relbar$\hfil\cr
  $\scriptstyle\vrule width0pt height.5ex\smash\rightarrow$\cr}}}}
\newcommand{\Rrelbar}{\mathrel{\raise.75ex\hbox{\oalign{%
  $\scriptstyle\relbar$\cr
  \vrule width0pt height.5ex$\scriptstyle\relbar$}}}}

\makeatletter
\def\leftrightarrowsfill@{\arrowfill@\leftrarrows\Rrelbar\lrightarrows}
\newcommand{\xleftrightarrows}[2][]{\ext@arrow 3399\leftrightarrowsfill@{#1}{#2}}
\makeatother

\makeatletter
\def\rightleftarrowsfill@{\arrowfill@\lrightarrows\Rrelbar\leftrarrows}
\newcommand{\xrightleftarrows}[2][]{\ext@arrow 3399\rightleftarrowsfill@{#1}{#2}}
\makeatother

\NewDocumentCommand{\Forall}{s}{
    \IfBooleanTF{#1}{
        \, \forall \,
    }{
        \text{ for all }
    }
}

\NewDocumentCommand{\rF}{
    O{R} d()
}{
    \IfNoValueTF{#2}{
        \mathbb{#1}
    }{
        \mathbb{#1}_{#2}
    }
} 
\NewDocumentCommand{\cF}{O{C}}{\mathbb{#1}} 

\NewDocumentCommand{\carprod}{m m}{#1 \times #2} 

\NewDocumentCommand{\ii}{s}{
    \IfBooleanTF{#1}{
        \mathbf{i}
    }{
        i
    }
}

\NewDocumentCommand{\ee}{sd()}{
    \IfBooleanTF{#1}{
        \IfNoValueTF{#2}{

        }{
            exp\left(#2\right)
        }
    }{
        \IfNoValueTF{#2}{
            e
        }{
            e^{#2}
        }
    }
}

\NewDocumentCommand{\tRe}{r()}{
    \mathrm{Re}\left(#1\right)
}
\NewDocumentCommand{\tIm}{r()}{
    \mathrm{Im}\left(#1\right)
}

\NewDocumentCommand{\sint}{sd()}{
    \IfBooleanTF{#1}{
        
    }{
        \mathrm{int}\,#2
    }
}
\NewDocumentCommand{\cl}{s d() O{cl}}{
    \IfBooleanTF{#1}{
        \bar{#2}
    }{
        \mathrm{#3}\,#2
    }    
}

\NewDocumentCommand{\setdiff}{d()d()}{%
    #1\backslash#2
}

\NewDocumentCommand{\matSp}{s O{C} m m}{
    \IfBooleanTF{#1}{
        #2^{#3\times #4}
    }{
        \mathbb{#2}^{#3\times #4}
    }
}

\NewDocumentCommand{\vecSp}{O{C} d()}{
    \IfNoValueTF{#2}{
        \mathbb{#1}
    }{
        \mathbb{#1}^{#2}
    }
} 

\NewDocumentCommand{\hilbertSp}{s o}{
    \IfBooleanTF{#1}{
        \IfNoValueTF{#2}{
            \mathbb{H}
        }{
            \mathbb{H}_{#2}
        }   
    }{
        \IfNoValueTF{#2}{
            \mathcal{H}
        }{
            \mathcal{H}_{#2}
        }   
    }
}

\NewDocumentCommand{\rank}{s r()}{
    \IfBooleanTF{#1}{
        \mathrm{rank}\,#2
    }{
        \mathrm{rank}\left(#2\right)
    }
}

\NewDocumentCommand{\rSp}{sd()}{
    \IfBooleanTF{#1}{
        \IfNoValueTF{#2}{
            \mathrm{range space}
        }{
            \mathrm{Ran}\left(#2\right)
        }
    }{
        \IfNoValueTF{#2}{
            \mathcal{R}
        }{
            \mathcal{R}\left(#2\right)
        }
    }
}

\NewDocumentCommand{\nr}{O{W}d()}{
    #1\left(#2\right)
}

\NewDocumentCommand{\anr}{O{W}d()}{
    #1'\left(#2\right)
}

\NewDocumentCommand{\maxph}{sO{\phi}d()}{
    \IfNoValueTF{#3}{
        \IfBooleanTF{#1}{
            #2_{\max}
        }{
            \bar{#2}
        }
    }{
        \IfBooleanTF{#1}{
            #2_{\max}\left(#3\right)
        }{
            \overline{#2}\left(#3\right)
        }
    }
}
\NewDocumentCommand{\minph}{sO{\phi}d()}{
    \IfNoValueTF{#3}{
        \IfBooleanTF{#1}{
            #2_\{\min\}
        }{
            \ubar{#2}
        }
    }{
        \IfBooleanTF{#1}{
            #2_\{\min\}\left(#3\right)
        }{
            \underline{#2}\left(#3\right)
        }
    }
}

\NewDocumentCommand{\Tr}{r()}{
    \mathrm{Tr}\,(#1)
}

\NewDocumentCommand{\hermp}{s O{H} r()}{
    \IfBooleanTF{#1}{
        (#3)_{\mathrm{#2}}
    }{
        #3_{\mathrm{#2}}
    }
}

\NewDocumentCommand{\shermp}{s O{S} r()}{
    \IfBooleanTF{#1}{
        \left(#3\right)_{\mathrm{#2}}
    }{
        #3_{\mathrm{#2}}
    }
}

\NewDocumentCommand{\mati}{s o}{
    \IfNoValueTF{#2}{
        \IfBooleanTF{#1}{
            \mathbb{I}
        }{
            I
        }
    }{
        \IfBooleanTF{#1}{
            \mathbb{I}_{#2}
        }{
            I_{#2}
        }
    }
} 
\NewDocumentCommand{\mato}{s}{
    \IfBooleanTF{#1}{
        \mathbf{0}
    }{
        0
    }
} 
\NewDocumentCommand{\toep}{s d()}{
    \IfBooleanTF{#1}{
        \mathbb{T}\,\mathrm{oep}(#2)
    }{
        \mathrm{Toep}\,(#2)
    }
}

\NewDocumentCommand{\inv}{d()}{
    \IfNoValueTF{#1}{
        ^{-1}
    }{
        \left(#1\right)^{-1}
    }
}

\NewDocumentCommand{\tp}{s d()}{
    \IfBooleanTF{#1}{
        \IfNoValueTF{#2}{
            ^{\mathsf{T}}
        }{
            \left(#2\right)^{\mathsf{T}}
        }
    }{
        \IfNoValueTF{#2}{
            ^\mathsf{T}
        }{
            \left(#2\right)^{\mathsf{T}}
        }
    }
}

\NewDocumentCommand{\ct}{s d()}{
    \IfBooleanTF{#1}{
        \IfNoValueTF{#2}{
            ^{\mathsf{H}} 
        }{
            \left(#2\right)^{\mathsf{H}}
        }
    }{
        \IfNoValueTF{#2}{
            ^{\mathsf{H}} 
        }{
            \left(#2\right)^{\mathsf{H}}
        }
    }
}
\NewDocumentCommand{\invct}{s d()}{
    \IfNoValueTF{#2}{
        \IfBooleanTF{#1}{
            ^{-\mathsf{H}}
        }{
            ^{-\mathsf{H}}
        }
    }{
        \IfBooleanTF{#1}{
            \left(#2\right)^{-\mathsf{H}}
        }{
            \left(#2\right)^{-\mathsf{H}}
        }
    }
}
\NewDocumentCommand{\pinv}{s d()}{
    \IfNoValueTF{#2}{
        \IfBooleanTF{#1}{
            ^{+}
        }{
            ^{\dagger}
        }
    }{
        \IfBooleanTF{#1}{
            \left(#2\right)^{+}
        }{
            \left(#2\right)^{\dagger}
        }
    }
}

\NewDocumentCommand{\iprod}{r() r() o}{
    \IfNoValueTF{#3}{
        \left \langle #1, #2 \right \rangle
    }{
        \left \langle #1, #2 \right \rangle_{#3}
    } 
}

\NewDocumentCommand{\cone}{D(){K}o}{
    \IfNoValueTF{#2}{
        \mathbf{#1}
    }{
        \mathcal{#1}#2
    }
}
\NewDocumentCommand{\psdcone}{sod()}{
    \IfBooleanTF{#1}{
        \IfNoValueTF{#2}{
            \IfNoValueTF{#3}{
                \mathcal{S}_{+}
            }{
                \mathcal{S}_{+}\left(#3\right)
            }
        }{
            \mathcal{S}_{+}^{#2}\left(#3\right)
        }
    }{
        \IfNoValueTF{#2}{
            \IfNoValueTF{#3}{
                \mathbf{PSD}
            }{
                \mathbf{PSD}_{#3}
            }
        }{
            \mathbf{PSD}_{#3}^{#2}
        }
    }
}
\NewDocumentCommand{\pdcone}{sod()}{
    \IfBooleanTF{#1}{
        \IfNoValueTF{#2}{
            \IfNoValueTF{#3}{
                \mathcal{S}_{++}
            }{
                \mathcal{S}_{++}\left((#3)\right)
            }
        }{
            \mathcal{S}_{++}^{#2}\left(#3\right)
        }
    }{
        \IfNoValueTF{#2}{
            \IfNoValueTF{#3}{
                \mathbf{PD}
            }{
                \mathbf{PD}_{#3}
            }
        }{
            \mathbf{PD}_{#3}^{#2}
        }
    }
}
\NewDocumentCommand{\ppsdcone}{sod()}{
    \IfBooleanTF{#1}{
        \IfNoValueTF{#2}{
            \IfNoValueTF{#3}{
                \mathcal{PS}_{+}
            }{
                \mathcal{PS}_{+}\left(#3\right)
            }
        }{
            \mathcal{PS}_{+}^{#2}\left(#3\right)
        }
    }{
        \IfNoValueTF{#2}{
            \IfNoValueTF{#3}{
                \mathbf{PPSD}
            }{
                \mathbf{PPSD}_{#3}
            }
        }{
            \mathbf{PPSD}_{#3}^{#2}
        }
    }
}
\NewDocumentCommand{\ppdcone}{sod()}{
    \IfBooleanTF{#1}{
        \IfNoValueTF{#2}{
            \IfNoValueTF{#3}{
                \mathcal{PS}_{++}
            }{
                \mathcal{PS}_{++}\left(#3\right)
            }
        }{
            \mathcal{PS}_{++}^{#2}\left(#3\right)
        }
    }{
        \IfNoValueTF{#2}{
            \IfNoValueTF{#3}{
                \mathbf{PPD}
            }{
                \mathbf{PPD}_{#3}
            }
        }{
            \mathbf{PPD}_{#3}^{#2}
        }
    }
}

\NewDocumentCommand{\pbcone}{s m d()}{
    \IfBooleanTF{#1}{
        \IfNoValueTF{#3}{
            \widetilde{\mathbf{SS}}#2
        }{
            \widetilde{\mathbf{SS}}_{#3}#2
        }
    }{
        \IfNoValueTF{#3}{
            \mathbf{SS}#2
        }{
            \mathbf{SS}_{#3}#2
        }
    }
}

\NewDocumentCommand{\spbcone}{s m d()}{
    \IfBooleanTF{#1}{
        \IfNoValueTF{#3}{
            \widetilde{\mathbf{S}}#2
        }{
            \widetilde{\mathbf{S}}_{#3}#2
        }
    }{
        \IfNoValueTF{#3}{
            \mathbf{S}#2
        }{
            \mathbf{S}_{#3}#2
        }
    }
}

\NewDocumentCommand{\ppbcone}{s m d()}{
    \IfBooleanTF{#1}{
        \IfNoValueTF{#3}{
            \mathbf{P\widetilde{SS}}#2
        }{
            \mathbf{P\widetilde{SS}}_{#3}#2
        }
    }{
        \IfNoValueTF{#3}{
            \mathbf{PSS}#2
        }{
            \mathbf{PSS}_{#3}#2
        }
    }
}

\NewDocumentCommand{\Cone}{s D(){K}}{
    \mathbf{#2}
}

\NewDocumentCommand{\compcone}{s m D(){\mathcal{G}}}{
    \IfBooleanTF{#1}{
        \widetilde{\mathbf{K}}_{#3} #2
    }{
        \mathbf{K}_{#3} #2
    } 
}

\NewDocumentCommand{\dcompcone}{s m D(){\mathcal{G}}}{
    \IfBooleanTF{#1}{
        \widetilde{\mathbf{D}}_{#3} #2
    }{
        \mathbf{D}_{#3} #2
    }
}

\NewDocumentCommand{\dcone}{s o d()}{
    \IfBooleanTF{#1}{
        \IfNoValueTF{#3}{
            ^{\circ}
        }{
            (#3)^{\circ}
        }
    }{
        \IfNoValueTF{#2}{
            \IfNoValueTF{#3}{
                ^{*}
            }{
                \left(#3\right)^{*}
            }
        }{
            \IfNoValueTF{#3}{
                ^{#2}
            }{
                \left(#3\right)^{#2}
            }
        }
    }
}

\NewDocumentCommand{\graph}{s m m}{
    \IfBooleanTF{#1}{
        \mathbb{G}\left(#2, #3\right)
    }{
        \left(#2, #3\right)
    }
}

\NewDocumentCommand{\gorder}{d()O{\nu}}{
    \IfNoValueTF{#1}{
        #2
    }{
        #2\left(#1\right)
    }
}
\NewDocumentCommand{\vset}{s O{V}}{
    \IfBooleanTF{#1}{
        \mathcal{#2}
    }{
        #2
    }
}
\NewDocumentCommand{\eset}{s O{E}}{
    \IfBooleanTF{#1}{
        \mathcal{#2}
    }{
        #2
    }
}

\NewDocumentCommand{\compg}{r()}{
    {#1}^{\mathrm{c}}
}

\NewDocumentCommand{\monoset}{s O{M} d()}{
    \IfBooleanTF{#1}{
        \mathbb{#2}\left(#3\right)
    }{
        \mathcal{#2}_#3
    }
}

\NewDocumentCommand{\matofgraph}{s O{M} D(){\mathrm{G}}}{
    \IfBooleanTF{#1}{
        \mathbb{#2}_{#3}
    }{
        \mathcal{#2}_{#3}
    }
}

\NewDocumentCommand{\cmatofgraph}{s O{C} D(){\mathrm{G}}}{
    \IfBooleanTF{#1}{
        \mathbb{#2}_{#3}^{n\times n}
    }{
        \mathcal{#2}_{#3}
    }
}

\NewDocumentCommand{\hmatofgraph}{s O{H} D(){\mathrm{G}}}{
    \IfBooleanTF{#1}{
        \mathbb{#2}_{#3}
    }{
        \mathcal{#2}_{#3}
    }
}

\NewDocumentCommand{\gph}{sO{G}d()}{\IfBooleanTF{#1}{\tilde{\mathcal{#2}}}{\mathcal{#2}}\IfNoValueTF{#3}{}{(#3)}}

\newcommand{\pogeq}{\succcurlyeq}

\NewDocumentCommand{\compset}{sd()o}{
    \IfBooleanTF{#1}{
        \IfNoValueTF{#2}{
            \mathcal{F}
        }{
            \IfNoValueTF{#3}{
                \mathcal{F}\left(#2\right)
            }{
                \mathcal{F}_{#3}\left(#2\right)
            }
        }
    }{
        \IfNoValueTF{#2}{
            \mathbb{F}
        }{
            \IfNoValueTF{#3}{
                \mathbb{F}\left(#2\right)
            }{
                \mathbb{F}_{#3}\left(#2\right)
            }
        }
    }
}

\newcommand{\setX}{\mathcal{X}}
\newcommand{\setY}{\mathcal{Y}}
\newcommand{\setZ}{\mathcal{Z}}

\NewDocumentCommand{\optimal}{sO{*}d()}{
    \IfNoValueTF{#1}{
        \left(#3\right)^{#2}
    }{
        #3^{#2}
    }
}

\NewDocumentCommand{\submat}{s m d()}{
    \IfBooleanTF{#1}{
        #2_{#3}
    }{
        #2[#3]
    }
}

\NewDocumentCommand{\cenComp}{s O{c} d()}{
    \IfBooleanTF{#1}{
            (#3)_{\mathrm{#2}}
        }{
            #3_{\mathrm{#2}}
        }
}

\newsiamthm{example}{Example}
\usepackage{algorithmicx}
\Crefname{ALC@unique}{Line}{Lines} 
\usepackage{algpseudocodex}
\usepackage{lipsum}
\usepackage{amsfonts}
\usepackage{graphicx}
\usepackage{subfig}
\usepackage{epstopdf}
\ifpdf
  \DeclareGraphicsExtensions{.eps,.pdf,.png,.jpg}
\else
  \DeclareGraphicsExtensions{.eps}
\fi

\newsiamremark{remark}{Remark}
\newsiamremark{hypothesis}{Hypothesis}
\crefname{hypothesis}{Hypothesis}{Hypotheses}
\newsiamthm{claim}{Claim}

\headers{Phase Bounded Completion and Decomposition}{D. Zhang, A. Ringh, and L. Qiu}

\title{Matrix Completion and Decomposition in Phase Bounded Cones
  \thanks{Submitted to the editors(TBD) DATE.
    \funding{This work was partially supported by the Hong Kong Research Grant Council under the project GRF 16201120, and by the Knut and Alice Wallenberg foundation via grant KAW 2018.0349.}
  }
}

\author{
  Ding Zhang\thanks{%
  Department of Electronic and Computer Engineering,
  Hong Kong University of Science and Technology,
  Clear Water Bay,
  Kowloon,
  Hong Kong,
  China (\email{ding.zhang@connect.ust.hk}).
  }%
  \and Axel Ringh\thanks{%
    Department of Mathematical Sciences,
    Chalmers University of Technology and the University of Gothenburg,
    SE-412 96 Gothenburg, Sweden  
    (\email{axelri@chalmers.se}).
  }
  \and Li Qiu\thanks{%
    School of Science and Engineering, 
    The Chinese University of Hong Kong, 
    Shenzhen, Guangdong, China (\email{qiuli@cuhk.edu.cn})
  }
}

\renewcommand{\descriptionlabel}[1]{\hspace{\labelsep}\textit{#1}} 

\ifpdf
\hypersetup{
  pdftitle={MatrixCompDecompPhase},
  pdfauthor={D. Zhang, A. Ringh, and L. Qiu}
}
\fi

\begin{document}

\maketitle

\begin{abstract}
The problem of matrix completion and decomposition in the cone of positive semidefinite (PSD) matrices is a well-understood problem, with many important applications in areas such as linear algebra,  optimization, and control theory. This paper considers the completion and decomposition problems in a broader class of cones, namely phase-bounded cones. We show that most of the main results from the PSD case carry over to the phase-bounded case. More precisely, this is done by first unveiling a duality between the completion and decomposition problems, using a dual cone interpretation. Based on this, we then derive necessary and sufficient conditions for the phase-bounded completion and decomposition problems, and also characterize all phase-bounded completions of a completable partial matrix with a banded pattern.
\end{abstract}



\section{Introduction}\vphantom{m}%
Let $\mathbb{M}$ be some ambient set of $n$-by-$n$ square matrices. In the ambient set $\mathbb{M}$, our focus often revolves around some prescribed subset $\Cone$ of matrices possessing a specific property, e.g., positive semidefiniteness, contractivity, total positivity, etc. 
An $n$-by-$n$ \emph{partial} matrix $C$ is a matrix with some of its entries given and others left undetermined. We can associate such a matrix $C$ with a graph $\mathcal{G} = \graph{\vset*}{\eset*}$ that describes the positions of its specified entries. Specifically, we say the partial matrix $C$ has pattern $\mathcal{G}$ if $C_{ij}$ is specified for all $i,j$ such that $\{i,j\}\in \eset*$ and all remaining entries of $C$ are undetermined. A \emph{matrix completion} problem then seeks to determine whether these unspecified entries can be chosen so that the completed matrix falls within the subset $\Cone$ of interest. 
In parallel, for a \emph{full} matrix $C \in \mathbb{M}$, we say it has a \emph{sparsity} pattern $\mathcal{G}$ if $C_{ij} = 0$ for all $i, j$ such that $\{i,j\} \notin \eset*$ while the remaining entries can be either zero or nonzero.  A \emph{matrix decomposition} problem seeks to decompose $C$ into a sum of matrices from $\Cone$, each with a sparsity pattern that is a sub-pattern of $\mathcal{G}$ so that the summands are sparser than the sum $C$.    

\subsection{PSD/PD completion and decomposition} \label{subsec:psd-survey}
In various contexts concerning the specific $\mathbb{M}$ and $\Cone$ of interest, this pair of problems has sparked enduring interests across, but not limited to, the linear algebra, optimization, and control theory communities. This has led to an extensive body of literature that continues to expand. Notably, one of the most prevalent instances is the positive semidefinite (PSD) completion and decomposition problems, where the ambient set $\mathbb{M}$ is the space $\vecSp[H]^n$ of Hermitian matrices (or real symmetric matrices) and $\Cone$ is the cone of (real) positive semidefinite matrices, denoted as $\psdcone$. In this case, the completion and decomposition problems are rather well-understood. 

On the completability side, if a partial matrix with pattern $\mathcal{G}$ can be completed to a positive semidefinite matrix, its specified principal submatrices are necessarily positive semidefinite.
This condition was first shown to be also sufficient when the graph $\mathcal{G}$ is banded \cite{dymExtensionsBandMatrices1981} (which, to be precise, works on the cone $\pdcone$ of positive definite matrices instead of $\psdcone$). The result was later generalized to chordal patterns in \cite{gronePositiveDefiniteCompletions1984}. 
For a $\psdcone$-completable partial matrix $C$ with block banded pattern, a linear fractional transform parameterization of all of its PSD completions is given in \cite{bakonyiCentralMethodPositive1992}.
These understandings naturally carry over to the positive definite cone. Moreover, for a $\pdcone$-completable partial matrix, a central completion which maximizes the determinant over all of its PD completions can be constructed, with an explicit formula derived in \cite{barrettDeterminantalFormulaeMatrix1989}. This central completion is also the unique PD completion whose inverse possesses the sparsity pattern $\mathcal{G}$ \cite{dymExtensionsBandMatrices1981}.
On the numerical front, the PSD completion problem is a special class of semidefinite programming problems, and its computational complexity is dicussed in detail in \cite{laurentPolynomialInstancesPositive2001}. In particular, let $\mathcal{G}=\graph{\vset*}{\eset*}$ be a graph with fixed minimum fill-in, i.e., a fixed minimum number of edges needed to be added to $\eset*$ so as to make $\mathcal{G}$ chordal. It is shown therein that the PSD completion of a partial rational matrix with pattern $\mathcal{G}$ can be solved in polynomial time. Beyond the problems involving chordal patterns, other stronger necessary conditions have been proposed for real symmetric partial matrices to be $\psdcone$-completable \cite{barrettRealPositiveDefinite1993, laurentRealPositiveSemidefinite1997}, where \cite{laurentRealPositiveSemidefinite1997} also shows the sufficiency of these conditions for partial matrices with a series-parallel pattern. 

While the approach to the $\psdcone$-completability in \cite{gronePositiveDefiniteCompletions1984} is constructive, an elegant conic duality links the completion problem to the PSD decomposition problem \cite{aglerPositiveSemidefiniteMatrices1988,paulsenSchurProductsMatrix1989}, offering a new way of proving the completability results in \cite{gronePositiveDefiniteCompletions1984}. Regarding the $\psdcone$-decomposability of a matrix, it was shown in \cite{hershkowitzPositiveSemidefinitePattern1990} for matrices with block tridiagonal sparsity pattern and in \cite{aglerPositiveSemidefiniteMatrices1988} for matrices with chordal sparsity patterns that a sparse Hermitian matrix can be decomposed into a sum of rank one PSD matrices, where the sparsity pattern of each rank one summand aligns with that of the original matrix, if and only if the matrix itself is PSD.  This decomposition has been extensively utilized to decompose large-scale sparse semidefinite problems into smaller ones, which gives rise to efficient distributed and parallel algorithms \cite{sunDecompositionConicOptimization2014,sunDecompositionMethodsSparse2015,zhengChordalFactorwidthDecompositions2021a}.

The PSD/PD completion and decomposition problems not only hold research interest on their own but are also closely related to completion/decomposition with respect to several other matrix properties. By exploiting these connections, results from the PSD/PD setup can be adapted to other types of problems.
A prominent example is the completion to the cone of Euclidean distance matrices (EDM). The EDM cone is the image of $\psdcone$ under a singular linear map (see, for example, \cite[Eqa.~11.2]{vandenbergheChordalGraphsSemidefinite2015}). The explicit connection between EDM and PSD completion problems is exposed in \cite{johnsonConnectionsRealPositive1995,bakonyiEuclidianDistanceMatrix1995}. As a consequence, analogous results hold for EDM completion problem, including the EDM-completability of a partial matrix with chordal pattern \cite{bakonyiEuclidianDistanceMatrix1995} and polynomial-time complexity of the EDM completion for a partial rational matrix whose pattern is described by graphs with fixed minimum fill-in \cite{laurentPolynomialInstancesPositive2001}. Insights from the EDM completion problem have found prolific applications in practical areas that involve distance geometry, such as sensor network localizations, molecular conformation, and phylogenetics \cite{alfakihSolvingEuclideanDistance1999,laurentMatrixCompletionProblems2001,soTheorySemidefiniteProgramming2007}. 
Many useful tools in control theory can be also linked to the PSD results. One example is that the block triangular contractive completion problem can be converted to a PSD completion problem with a banded structure \cite{qiuContractiveCompletionBlock1996,qiuUnitaryDilationApproach2004}. 
Another example is that the useful Parrots' theorem \cite[{Thm.~2.21}]{zhouRobustOptimalControl1995} can be understood from a PSD completion viewpoint. Both examples are due to the equivalence: $\rho\mati - \left[\begin{smallmatrix}\mato & A \\ A\ct* & \mato\end{smallmatrix}\right] \in \psdcone \Leftrightarrow \|A\|_2 \leq \rho$.
A third example is the completion and decomposition of Sum-of-Squares (SOS) matrices, which is of particular interest in system and control theory \cite{chesiDomainAttraction2011}. Authors in \cite{zhengDecompositionCompletionSumofSquares2018,zhengSumofsquaresChordalDecomposition2023} characterize the completability of partial SOS matrices with a chordal pattern and the decomposability of a subset of SOS matrices with a chordal sparsity pattern by lifting SOS matrices to their PSD block matrix representations.

Our review here is in no way complete.
It is only meant to expose some of the results that are most relevant to the context of this paper,
in which we add another member to this family of completion and decomposition problems that are closely tied to the PSD/PD ones.
Readers may consult \cite{laurentTourHorizonPositive1998,bakonyiMatrixCompletionsMoments2011,vandenbergheChordalGraphsSemidefinite2015,zhengChordalFactorwidthDecompositions2021a} and references therein for more comprehensive treatments on PSD/PD completion and decomposition.

\subsection{Problem formulation in phase-bounded cones} \label{subsec:pb-form}
In recent years, the concept of matrix phases has been revitalized in the study of multivariable systems \cite{wangPhasesComplexMatrix2020,chenPhaseTheoryMultiInput2024,maoPhasesDiscretetimeLTI2022}.
The quantitative phase measure plays an equally pivotal role therein as the rather mature gain-based (a.k.a singular value based) analysis and it offers nice physical interpretation when applied to specific physical systems, e.g., the cosine of the largest absolute phase of an impedance matrix corresponds to the worst-case power factor of its associated circuit realization \cite{chenPhasePreservationNport2023}.
 A more detailed exploration of matrix phases will be provided in Section 3, but for now, let us accept that a nonzero square complex matrix of rank $r$ will either have no phase or $r$ phases, and in the latter case they all reside in the closed real interval of length no greater than $\pi$.
In phase-based analysis, the cone of phase bounded matrices, denoted by $\pbcone{[\alpha,\beta]}$ where $\alpha, \beta \in \rF$ serve as lower and upper bounds on matrix phases, naturally arises and becomes the major object of interest. 
This motivates us to consider the completion problem and the decomposition problem in phase bounded cones. 
In this context, the ambient set $\vecSp[M]$ becomes $\mathbb{C}^{n\times n}$ and the property $\Cone$ instantiates as $\pbcone{[\alpha,\beta]}$. Our focus is on the following two questions:
\begin{enumerate}
    \item Given a partial matrix $C$ with pattern $\mathcal{G}$, what conditions on specified entries of $C$ and the graph $\mathcal{G}$ allow us to complete $C$ to a matrix in $\pbcone{[\alpha,\beta]}$?
    \item Given a full matrix $C$ with sparsity pattern $\mathcal{G}$, what conditions must $C$ satisfy to be decomposable as $C_1 + \cdots + C_m$, where each $C_i\in \pbcone{[\alpha,\beta]}$ adheres to a sparser pattern than $\mathcal{G}$ for all $i=1,\ldots,m$?
\end{enumerate} We will provide answers to these questions when $\mathcal{G}$ is chordal, and we will also consider the parameterization of all phase-bounded completions of a $\pbcone{[\alpha,\beta]}$-completable partial matrix.
More specifically, for phase-bounded matrices, we show that under an appropriately defined inner product space, the cones related to the completion and decomposition problems also form two pairs of dual cones.
Using these cones, we can show that the duality between these two problems is preserved in phase bounded cones. We give necessary and sufficient conditions for completion and decomposition problems when the underlying graph is chordal. Finally, we also discuss the characterization of all completions of a completable matrix with a banded pattern.
Note that the PSD cone can be interpreted as the cone $\pbcone{[0]}:=\pbcone{[0,0]}$ of matrices with both its lower and upper phase bounds being zero, and therefore our results can be interpreted as a generalization of PSD results to a more general cone of matrices.
        
\subsection{Outline} \label{subsec:outline}
The remainder of the paper is organized as follows. In \cref{sec:prelim}, we present the preliminaries on graph theory, with a particular focus on chordal graphs, as well as on convex cones. \Cref{sec:phases} presents the matrix phases and establishes some elementary properties of phase-bounded cones. \Cref{sec:comp-dcomp-cond} covers our main results, where we introduce four important cones associated with the phase-bounded matrix completion and decomposition problems. We show the duality among these cones and discuss their relationships. This, in the case of chordal graphs, gives the necessary and sufficient conditions of matrix completion and decomposition problems in phase-bounded cones. 
We also briefly address the phase-bounded completion problem of partial matrices with an asymmetric structure towards the end of the section.
Following this, \cref{sec:characterization} provides a characterization of all phase-bounded completion of a completable partial matrix with a banded pattern.
Finally, we conclude the paper in \cref{sec:concl}.

\section{Preliminaries}
\label{sec:prelim}
In this section, we first set up the general notations used in this paper. Then we build up some preliminary concepts and relevant properties from graph theory and convex analysis.

\subsection{Nomenclature} \label{subsec:notation}
Throughout the text, $\rF[F] = \rF \text{ or } \mathbb{C}$ denotes the field of real or complex numbers.
We refer to the set of positive integers as $\rF[N]_+$.
We denote by $A\tp, A\ct$ the transpose and the conjugate transpose of a matrix $A
\,\in\mathbb{F}^{m\times n}$, respectively. 
Moreover, $\rSp*(A)$ denotes the range of the matrix. 
For two matrices $A_1 \in\mathbb{F}^{m_1\times n_1}$ and $A_2 \in\mathbb{F}^{m_2\times n_2}$, $A = A_1 \oplus A_2$ denotes the direct sum of the matrices, i.e., $A \in \mathbb{F}^{(m_1 + m_2)\times (n_1 + n_2)}$ is the $2 \times 2$ blockdiagonal matrix with $A_1$ and $A_2$ on the diagonal.
For a set $S \subset \vecSp[M]$, the set $A\inv S$ refers to the preimage of $S$ under the linear map $A$, i.e., $A\inv S = \{x\in \mathbb{M} \mid A x \in S\}$.
For a square matrix $A \in \rF[F]^{(n+1)\times (n+1)}$, the diagonal elements $A_{ii}$ compose the \emph{main diagonal}. While the main diagonal is also indexed as the zeroth diagonal, the diagonals below it are indexed from $-1$ to $-n$ sequentially, and the diagonals above it are indexed from $1$ to $n$.
A Toeplitz matrix is a matrix with elements in each of its diagonals  being the same. Hence, we can represent it with its $(2n+1)$ diagonal constants $a_{-n},\ldots,a_{-1},a_{0},a_{1},\ldots,a_{n}\in \mathbb{F}$.
And we shall use $\toep(a_{-n},\ldots,a_{-1},\boxed{a_0},a_{1},\ldots,a_{n})$ to denote a Toeplitz matrix where the box indicates the constant for the zeroth diagonal.
In the space $\vecSp[H]^n$ of $n$-by-$n$ Hermitian matrices, 
we write $H\pogeq \mato$ if $H\,\in \vecSp[H]^n$  falls within $\psdcone$.
For a complex scalar $z\in \mathbb{C}$, we use $\tRe(z)$ and $\tIm(z)$ to refer to its real part and its imaginary part.
For two sets $S_1, S_2$,
        we denote by $S_1+S_2$ the Minkowski sum of $S_1$ and $S_2$, i.e.,
        $S_1+S_2:=\{x+y \mid x\in S_1, y\in S_2\}$. By $\setdiff(S_1)(S_2)$, we mean the set difference, that is $\setdiff(S_1)(S_2):=\{x \in S_1\mid x\notin S_2\}$.
The product $S_1 \times S_2$ of sets $S_1, S_2$ refers to their Catersian product $\{(x,y)\mid x\in S_1, y\in S_2\}$. For convenience, we also use the matrix form $\left[\begin{smallmatrix}x \\ y\end{smallmatrix} \right]$ to represent $(x,y) \in S_1 \times S_2$ if $x,y$ are matrices with the same number of columns.

\subsection{Chordal graphs}\label{subsec:graph}
As mentioned in the introduction, graphs are used to describe the sparsity pattern (or pattern) of (partial) matrices. This section aims to formalize some graph-related concepts and notations. We primarily focus on an undirected graph 
    $\mathcal{G} =\graph{\vset*}{\eset*}$
    where $\vset*$ is the set of vertices and $\eset*\subseteq \carprod{\vset*}{\vset*}$ is the set of edges. The cardinality of the vertex set is denoted as $|\vset*|$. The edge set $\eset*$ of an undirected graph is symmetric. Thus, an unordered pair $\{u,v\}$ is used to denote an undirected edge joining vertices $u$ and $v$. We assume that each vertex has a self-loop, i.e., $\{u,u\}\in \eset*$ for all $u \in \vset*$.
For graph $\mathcal{G}$, we define its complement graph as $\compg(\mathcal{G}):=\graph{\vset*}{\compg(\eset*)}$ where $\compg(\eset*)=\setdiff((\vset*\times \vset*))(\eset*)$.
Let $\vset*[W] \subseteq \vset*$ be a subset of vertices. This set induces a 
    \emph{subgraph $\mathcal{G}(\vset*[W]):=\graph{\vset*[W]}{\eset*(\vset*[W])}$}
        where $\eset*(\vset*[W]) := \left\{
            \{a,b\}\in \eset* \mid a, b \in \vset*[W]
        \right\}$. 
We use $\tilde{\mathcal{G}}(\vset*[W])$ to denote another analogous graph $\graph{\vset*}{\eset*(\vset*[W])}$ but with the vertex set $\vset*$ preserved.
Unless otherwise defined, by putting a graph $\mathcal{G}$ at the subscript of a matrix set $\vecSp[M]$, i.e., $\hmatofgraph*[M](\mathcal{G})$, we are referring to the set of matrices in $\vecSp[M]$ that have sparsity pattern $\mathcal{G}$.

A pivotal class of undirected graphs in our context are \emph{chordal graphs}.
They exhibit nice algebraic and algorithmic properties and play an important role in distributed control and sparse matrix computation. Here we introduce the relevant properties only, and refer readers to \cite{blairIntroductionChordalGraphs1993,vandenbergheChordalGraphsSemidefinite2015} for an expository introduction.
    
For an undirected graph $\mathcal{G} = \graph{\vset*}{\eset*}$, a \emph{simple cycle} of length $m$ refers to a sequence $(v_1,\ldots,v_m,v_1)$ of consecutively adjacent vertices which are distinct except the duplication of the first and the last vertices. If two nonconsecutive vertices $v_i,v_j$, i.e., $ 2\leq |j-i| < m-1$, in the cycle are adjacent in $\mathcal{G}$, then the edge $\{v_i,v_j\}$ is called a \emph{chord} of the cycle. 
Finally, the graph $\mathcal{G}$ is said to be \emph{chordal} if every simple cycle of length no less than 4 has a chord. 

A set $\vset*[K] \subseteq \vset*[V]$ of vertices is called a \emph{clique} 
    if the subgraph induced by $\vset*[K]$, i.e., $\mathcal{G}(\vset*[K])$, is complete. 
    A clique is \emph{maximal} if it is not properly contained in any other clique of the graph $\mathcal{G}$.
For a graph $\mathcal{G}$ with $n$ vertices,
an ordering $\gorder$ of the graph is a bijection from $\{1,\ldots,n\}$ to $\vset*$.
For notational simplicity and clarity, we assume that each graph $\mathcal{G}$ has a default ordering $\nu_0$. Unless other orderings are explicitly mentioned, we shall use indices to refer to vertices ordered by $\nu_0$, as we already did in the introduction.
Hence, the notation $\mathcal{K}$ for a clique can also be interpreted as the index set $\{\nu_0\inv\,(u): u\in \mathcal{K}\}$ of the clique when appropriate.
Under an ordering $\gorder$, 
    we can define the $i$-th \emph{monotone vertex set} $\monoset(i)$ as
        $\monoset(i):= \left\{ \gorder(j) \mid j\geq i \right\}$ where $i=1,\ldots,n$. 
    Clearly, we have $\monoset(1) = \vset*$.  
An ordering $\gorder$ is said to be a \emph{perfect elimination ordering (PEO)}
    if, for each $i \in \{1,\ldots,n\}$,
        the neighbors of $\nu(i)$ in $\mathcal{M}_i$ form a clique.
There are multiple ways to characterize chordality and the following lemma is of technical necessity for later developments.
\begin{lemma}[PEO characterization {\cite[{Thm.~2.3}]{blairIntroductionChordalGraphs1993}}] \label{lem:chordal-peo}
    An undirected graph $\mathcal{G}$ is chordal if and only if it has a perfect elimination ordering.
\end{lemma}

Let $\vset*[I]\subseteq \{1,\ldots,n\}$ be an index set. An $n$-dimensional vector $x$ is said to be an \emph{$\vset*[I]$-supporting vector} if $x_i = 0$ for all $i\notin \vset*[I]$. Let $\vset*[J]$ be another index set. Then for a square matrix $C$, we use $\submat{C}(\vset*[I],\vset*[J])$ to represent the submatrix of $C$ with row indices $\vset*[I]$ and column indices $\vset*[J]$. When $\vset*[I] = \vset*[J]$, we simplify the notation to $\submat{C}(\vset*[I])$, denoting the principal submatrix indexed by $\vset*[I]$.
For clarity, we assume without loss of generality that all index sets are ordered increasingly. This is due to the fact that the phase property, which will be detailed in \cref{sec:phases}, is invariant under congruence transformations, including the congruence via permutation matrices.
Now, suppose that matrix $C \in \matofgraph*(\mathcal{G})$, and that $\mathcal{K}\subseteq \vset*$ is a clique of $\mathcal{G}$. We then refer to the matrix $C[\mathcal{K}]$ as a \emph{clique supporting matrix} or \emph{$\mathcal{K}$-supporting matrix} to be concrete.

\subsection{Cones and dual cones} \label{subsec:prelim-cones}%
Let $\mathbb{X}$ be a (finite-dimensional) Hilbert space with inner product $\iprod(\cdot)(\cdot)\in \mathbb{F}$, which we refer to as $(\mathbb{X},\iprod(\cdot)(\cdot))$. 
For a subspace $\mathbb{S}$ of $\mathbb{X}$, its orthogonal complement $\mathbb{S}^\perp$ is defined as $\mathbb{S}^\perp :=\left\{x\in \mathbb{X}\mid \iprod(x)(y)=0 \Forall y\in \mathbb{S}\right\}$.  

A nonempty set $S\subseteq \mathbb{X}$ is said to be a cone if it is closed under 
positive scaling.
A cone is \emph{pointed} if it contains no nontrivial subspace. 
It is \emph{fully dimensional} if the interior of the cone is nonempty.

When the ambient space $\mathbb{X}$ is real, the dual cone $S\dcone*$ of $S$ is defined as
    $S\dcone*:= 
        \left\{
            x \in \mathbb{X}\mid \iprod(x)(y)\geq 0 \text{ for all } y\in S
        \right\}$. 
By definition, the dual cone is always closed and convex \cite[{Prop.~6.24}]{bauschkeConvexAnalysisMonotone2017}.
We say a cone $S$ is \emph{self-dual} if $S=S\dcone*$.

Now let us consider a collection of real Hilbert spaces $(\vecSp[X]_i,\iprod(\cdot)(\cdot)[i])$ with $i=1,\ldots,m$. Their product $\vecSp[X]:=\vecSp[X]_1 \times \cdots \times \vecSp[X]_m$ is a vector space under the natural scalar multiplication and addition: $\lambda (x_1,\ldots, x_m) = (\lambda x_1, \ldots, \lambda x_m)$ and $(x_1,\ldots, x_m)+(y_1,\ldots,y_m) = (x_1+y_1,\ldots, x_m+y_m)$ for all $\lambda \in \rF$, $x_i, y_i \in \vecSp[X]_i$ and $i=1,\ldots,m$.
It holds that:
\begin{fact}[Dual cone in product space] \label{clm:dual-product}
    The product space $\vecSp[X]$ is a Hilbert space when equipped with the binary operation $\iprod((x_1,\ldots,x_m))((y_1,\ldots,y_m))[\vecSp[X]]:= \sum_{i=1}^{m}\iprod(x_i)(y_i)[i]$ for all $(x_1,\ldots,x_m)$, $(y_1,\ldots, y_m) \in \vecSp[X]$; Let $\Cone_i$ be a cone in $\vecSp[X]_i$ for $i=1,\ldots,m$, 
        then we have $\dcone*(\Cone_1\times \cdots \times \Cone_m)= \Cone_1\dcone* \times \cdots \times \Cone_m\dcone*$ in $(\vecSp[X],\iprod(\cdot)(\cdot)[\vecSp[X]])$.
\end{fact}
    
A crucial idea in obtaining our results is to exploit the connection between the phase-bounded cone and the well-studied positive semidefinite cone. This, as will be discussed in \cref{sec:phases} later, can be achieved via an invertible linear map.
The following observation reveals that if a cone can be mapped to a self-dual cone in a Hilbert space via an invertible linear map, then we can always endow the space, where the cone resides, with a structure so that the cone itself becomes self-dual.
\begin{lemma}[Self dualization] \label{clm:self-dualization}
    Consider a real vector space $\vecSp[X]$ and a real Hilbert space $(\vecSp[Y],\iprod(\cdot)(\cdot)[\vecSp[Y]])$. 
    Let $T:\vecSp[X]\mapsto \vecSp[Y]$ be an invertible linear map and define accordingly a binary operation $\iprod(x_1)(x_2)[T\vecSp[X]]:= \iprod(Tx_1)(Tx_2)[\vecSp[Y]]$ for all $x_1, x_2 \in \vecSp[X]$, then:
    \begin{enumerate}
        \item The space $\vecSp[X]$ equipped with $\iprod(\cdot)(\cdot)[T\vecSp[X]]$ is a Hilbert space;
        \item For any closed set $Y$ contained in $(\vecSp[Y],\iprod(\cdot)(\cdot)[\vecSp[Y]])$, the set $T\inv Y$ is a closed set in  $(\vecSp[X],\iprod(\cdot)(\cdot)[T\vecSp[X]])$; 
        \item For any cone $\Cone$ contained in $\vecSp[Y]$, we have $\dcone*(T\inv \Cone) = T\inv \Cone\dcone*$. Consequently, if $\Cone$ is self-dual in $(\vecSp[Y],\iprod(\cdot)(\cdot)[\vecSp[Y]])$, then $T\inv \Cone$ is also self-dual in $(\vecSp[X],\iprod(\cdot)(\cdot)[T\vecSp[X]])$.
    \end{enumerate}  
\end{lemma}
\begin{proof}
    See \cref{pf:self-dualization}.
\end{proof}

\section{From positive semidefinite cone to phase-bounded cones} \label{sec:phases}
Recently, the notion of \emph{matrix phases} was revitalized in \cite{wangPhasesComplexMatrix2020}, which sheds light on the long missing phase part in multivariable system theory \cite{chenPhaseTheoryMultiInput2024, maoPhasesDiscretetimeLTI2022}.
For a square complex matrix, while singular values are well accepted as the matrix gains, the matrix phases are built around the concepts of numerical range and canonical angles \cite{hornEigenvaluesUnitaryPart1959,hornTopicsMatrixAnalysis1994}. 
    
\begin{figure}[tbhp]
    \centering 
    \subfloat[sectorial $C_1$]{\makebox[.21\columnwidth][c]{\includegraphics[height=4.5cm]{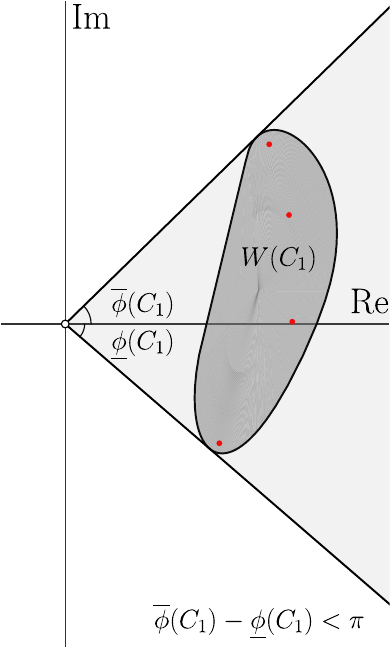}} \label{fig-sect}}
    \subfloat[quasisectorial $C_2$, semisectorial $C_3$]{\makebox[.4\columnwidth]{\includegraphics[height=4.5cm]{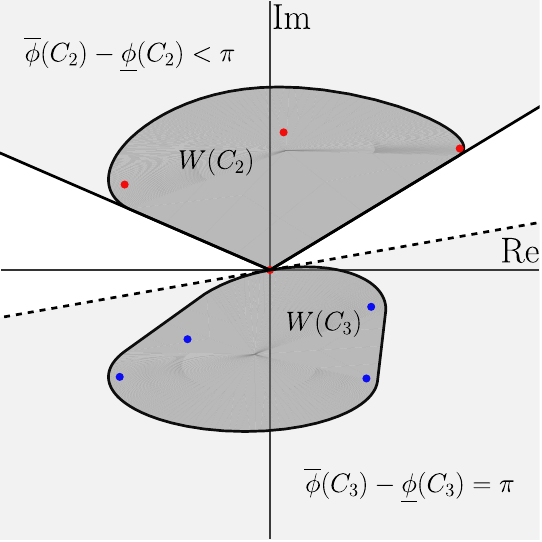}} \label{fig-qs-ss}}
    \subfloat[non-semisectorial $C_4$]{\makebox[.35\columnwidth]{\includegraphics[height=4.5cm]{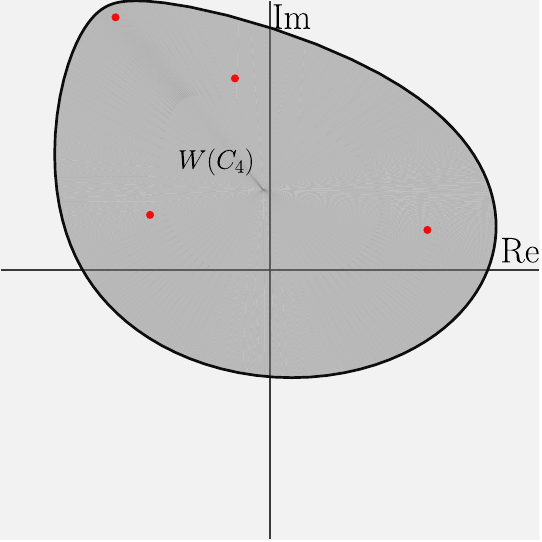}} \label{fig-nss}}
    \caption{Numerical range (shaded in dark gray) and angular numerical range (shaded in gray) of
        four matrices $C_1, C_2, C_3, C_4 \in \matSp{4}{4}$, along with their eigenvalues (dotted).} \label{fig-nr-anr-demo}
\end{figure}

For a square complex matrix $C \in \matSp{n}{n}$,
    its \emph{numerical range} $\nr(C)$ is defined to be the set
    $\left\{ x\ct* C x \in \mathbb{C} \mid x \in \mathbb{C}^n, x\ct* x = 1 \right\}$, which is known to be convex and compact \cite[Sec.~1.2]{hornTopicsMatrixAnalysis1994}. 
The \emph{angular numerical range} 
    $\anr(C):=\left\{ x\ct* C x \in \mathbb{C}\mid x \in \mathbb{C}^n\backslash \{0\}\right\}$ is the smallest cone that contains $\nr(C)$.
Both $\nr(C)$ and $\anr(C)$ serve as 2D graphical representation of matrix $C$. Though the correspondence is in general not one-to-one, one can indeed read off rich information about $C$ from them.
For example, $\nr(C)$ contains all eigenvalues of $C$ and is exactly the convex hull of eigenvalues when $C$ is normal \cite[{Sec.~1.2}]{hornTopicsMatrixAnalysis1994}. 
In particular, matrix $C$ is Hermitian if and only if $\nr(C)$ is a line segment on the real axis.
    A Hermitian $C$ is also normal, and thus the two endpoints of $\nr(C)$ are given by its smallest and largest eigenvalues.
    It follows that the union of the numerical ranges of all matrices in $\psdcone$ is exactly the nonnegative real axis.

A matrix $C$ is said to be \emph{semisectorial} if the origin is not contained in the interior of its numerical range, i.e., $0\notin \sint(\nr(C))$. The angular numerical range of a semisectorial matrix (see \cref{fig-qs-ss}) is a convex cone contained in a closed half-plane, while that of a non-semisectorial matrix is the whole complex plane (see \cref{fig-nss}).
A semisectorial $C$ is called \emph{quasisectorial} if the closure of its angular numerical range, denoted $\cl(\anr(C))$, is pointed. Furthermore, a quasisectorial $C$ is said to be \emph{sectorial} if the origin is excluded from its numerical range, i.e., $0\notin \nr(C)$ (see \cref{fig-sect}). The classes of sectorialness, quasisectorialness, and semisectorialness are successively inclusive.

It is well-known that the matrix eigenvalues are invariant under similiarity transformation. In particular, they are given by the diagonal elements of the similarity canonical form, namely the Jordan form.
Analogously, the matrix phases are the congruence\footnote{%
Throughout this paper, by ``congruence'', we are always referring to the $*\,$-congruence. In other words, a matrix $A$ is congruent to $B$, denoted by $A\simeq B$,  if there exists some nonsingular $T$ such that $A = TBT\ct*$. Congruence defines an equivalence relation.} invariants for semisectorial matrices, and are related to the diagonal elements of the congruence canonical form. The study on this non-trivial form traces back to \cite{furtadoSpectralVariationCongruence2001,furtadoSpectralVariationCongruence2003} for the class of semisectorial matrices, and was later extended to general matrices \cite{hornRegularizationAlgorithmMatrices2006,hornCanonicalFormsComplex2006}.
It can be deduced from \cite[Theorem~5]{furtadoSpectralVariationCongruence2003} that the congruence canonical form of a semisectorial~$C$ is a direct sum of three possible types of canonical blocks:

\begin{align} \label{eqa-congruence-canonical}
    C \simeq 
    \left(
        \mato*_k
    \right)
    \oplus
    \left(
        \bigoplus_{j=1}^{l} e^{\ii* \alpha_j}
    \right)
    \oplus
    \left(
        \bigoplus_{j=1}^{m}
        e^{\ii* \gamma} 
        \begin{bmatrix}
            1 & 2\\
            0 & 1 
        \end{bmatrix}
    \right)
\end{align}
where  $k, l, m\geq 0, n = k + l + 2m$, and $\alpha_1,\ldots,\alpha_l, \gamma \in \rF$.
Note that each block may or may not appear. The arguments $\alpha_1,\ldots,\alpha_l,\gamma$ are all multi-valued but unique modulo $2\pi$.
From \cref{eqa-congruence-canonical}, we observe that 0 is always a semi-simple eigenvalue with multiplicity $k$ for a singular semisectorial matrix of rank $l+2m$.
It is also noteworthy that the numerical range of $e^{\ii*\gamma}[\begin{smallmatrix} 1 & 2 \\ 0 & 1\end{smallmatrix}]$ is a disk centered at $e^{\ii*\gamma}$ with the origin on its boundary \cite{liSimpleProofElliptical1996}. Consequently, its angular numerical range consists of the union of an open half-plane subtended by angles $\gamma-\frac{\pi}{2}, \gamma+\frac{\pi}{2}$ and the origin.

By definition, the angular numerical range $\anr(C)$ is invariant under congruence. Additionally, we have the identity that $\anr(C_1 \oplus C_2)$ equals the convex hull of $\anr(C_1)\cup\anr(C_2)$. Thereby, for a specific semisectorial $C$, though not all arguments $\alpha_1,\ldots,\alpha_l, \gamma\pm\frac{\pi}{2}$ from the general form~\cref{eqa-congruence-canonical} may be present, the values for those present can always be settled in an interval of length less than or equal to $\pi$.
\begin{definition}[Matrix phases] \label{def-matrix-phases}
    Let $C \in \matSp{n}{n}$ be a semisectorial matrix of rank~$r$. 
    In light of form~\cref{eqa-congruence-canonical} 
        and the discussion above, 
            we refer to any choice of values within a $\pi$-interval for
                the $r$ members of $\alpha_1,\ldots,\alpha_l$ and $m$ copies of $\gamma\pm \frac{\pi}{2}$ that appear in the congruence canonical form of $C$ as a set of phases of $C$. 
            For a set of phases of $C$, we arrange them in nonincreasing order and denote them as
            \begin{align*}
                \maxph(C) := \phi_1(C) \geq \cdots \geq \phi_r(C) =: \minph(C)
            \end{align*} where $\minph(C),\maxph(C)$ refer to the smallest and the largest phases of $C$, respectively.
            For the zero matrix, we let $\minph(\mato) = \infty$ and $\maxph(\mato) = -\infty$.
\end{definition}
\begin{remark} \label{rmk-def-1}
    The rays defined by $\maxph(C)$ and $\minph(C)$ are exactly the extreme rays of the angular numerical range $\anr(C)$ as depicted in \cref{fig-nr-anr-demo}, i.e., $\minph(C) \text{ and } \maxph(C)$ completely determine $\cl(\anr(C))$.
    If $m>0$, obviously $\maxph(C)-\minph(C) = \pi$ as $\cl(\anr(C))$ is a closed half-plane. Recall that $\cl(\anr(C))$ of a quasisectorial $C$ is pointed, i.e., $\maxph(C)-\minph(C)<\pi$. This implies that $m$ must be zero for quasisectorial matrices.
\end{remark}

We refer readers to \cite{wangPhasesComplexMatrix2020,WANG2023441,chenPhaseTheoryMultiInput2024} for more detailed discussions on matrix phases.
In this work, we consider the cone of semisectorial matrices with phases bounded by $\alpha \text{ and } \beta$, which is denoted by $\pbcone{[\alpha,\beta]}$:
\begin{align} \label{eqa:pb-cone}
    \pbcone{[\alpha,\beta]} := \left\{
        C \in \matSp{n}{n}\mid
        C \text{ is semisectorial},
        \alpha \leq \minph(C) \leq \maxph(C) \leq \beta
    \right\}.
\end{align}
In the notation $\pbcone{[\alpha,\beta]}$, we have omitted the size of the matrices, which should be clear from its context.
We require that $0\leq \beta-\alpha \leq \pi$ such that $\pbcone{[\alpha,\beta]}$ is convex.
For $\beta-\alpha=0$, the cone reduces to the essentially positive semidefinite cone, i.e., $e^{-\ii*\alpha} \pbcone{[\alpha,\beta]}$ is exactly $\psdcone$. In this case, the cone is not full-dimensional in its ambient space.
For $\beta-\alpha=\pi$, $\pbcone{[\alpha,\beta]}$ becomes the cone of \emph{essentially accretive matrices}%
\footnote{
    A square matrix $C$ is accretive if $C+C\ct*$ is positive semidefinite. It is essentially accretive if there exists $\alpha\in (-\pi,\pi]$ such that $e^{\ii*\alpha} C$ is accretive. Note that there is ambiguity in defining accretive matrices and we align our definition with \cite[{p.~279}]{katoPerturbationTheoryLinear1995}, \cite[{p.~4}]{chenPhaseTheoryMultiInput2024}.
}.
This cone is not pointed as it contains the nontrivial subspace 
    $\{e^{\ii* \alpha}H \mid H \in \vecSp[H]^n\}$. 
Finally, for $0<\beta-\alpha<\pi$, which is the case we are mainly interested, the cone is convex, pointed, and full-dimensional.
It can be readily recognized, from \cref{def-matrix-phases} and remarks thereafter, that the union of numerical ranges of all matrices in $\pbcone{[\alpha,\beta]}$ is exactly the closed convex cone substented by angles $\alpha,\beta$ in the complex plane (see \cref{fig:nr-union-cone}).

\begin{figure}[tbhp]
    \centering 
    \subfloat[PSD cone]{\label{fig:nr-union-psd}
    \makebox[.3\columnwidth]{
        \includegraphics[height=3.8cm]{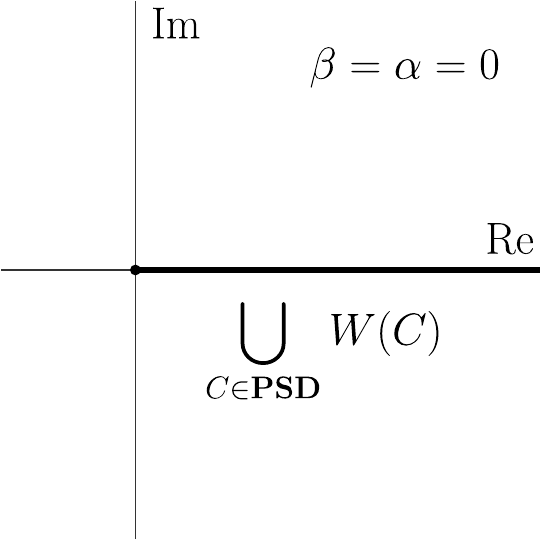}
        }
    } 
    \subfloat[quasisectorial cone]{\label{fig:nr-union-quasi}
        \makebox[.3\columnwidth]{
            \includegraphics[height=3.8cm]{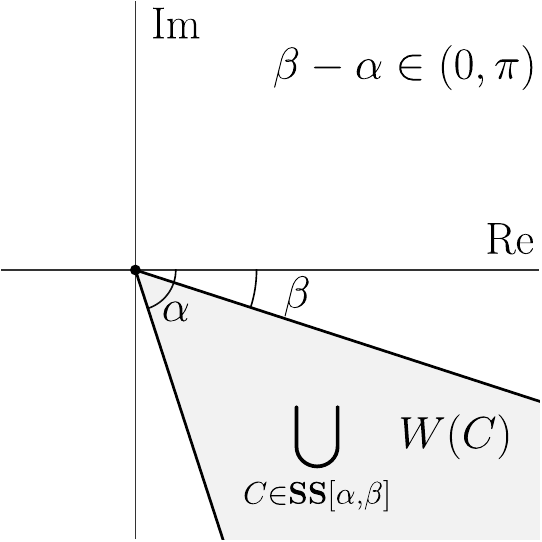}
        }
    } 
    \subfloat[accretive cone]{\label{fig:nr-union-porealsemi}
        \makebox[.3\columnwidth]{
            \includegraphics[height=3.8cm]{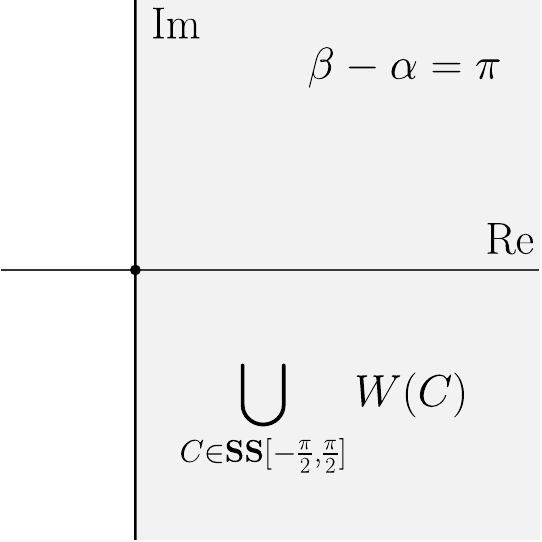}
        }
    }
    \caption{The union of numerical ranges of matrices in phase-bounded cones.} \label{fig:nr-union-cone}
\end{figure}

We will now provide a PSD characterization of phase-bounded cone through the abovesaid graphical intuition and a realification of $\mathbb{C}^{n\times n}$. To this end, first note that, in general, one can either realify $C \in \mathbb{C}^{n\times n}$ via the elementwise Cartesian decomposition $C = \tRe(C)+\ii*\,\tIm(C)$ where $\tRe(C),\tIm(C)\in \rF^{n\times n}$, or via the Toeplitz decomposition $C = \hermp(C)+\ii* \shermp(C)$ where $\hermp(C) = \frac{1}{2}(C+C\ct*), \shermp(C) = \frac{1}{2\ii*}(C-C\ct*)\,\in \vecSp[H]^n$.%
\footnote{
    Note that $\hermp(C)$ is usually referred to as the Hermitian part of $C$ while $\shermp(C)$ is slightly different from the skew-Hermtian part $\frac{1}{2}\left(C-C\ct* \right)$ defined in literature
    (see \cite[{p.~9}]{hornTopicsMatrixAnalysis1994} and \cite{filipequeiroCartesianDecompositionMatrix1985} for example).
    We divide the skew-Hermitian part by $\ii*$ such that $\shermp(C)$ is Hermitian. 
}
Here, we shall adopt the latter due to its clear connection with the numerical range of $C$. To be specific, the orthogonal projections of $\nr(C)$ onto the real and imaginary axes are exactly $\nr(\hermp(C))$ and $\ii*\,\nr(\shermp(C))$ \cite[Thm.~1.2.5]{hornTopicsMatrixAnalysis1994}. 
For a set $S\subseteq \mathbb{C}^{n\times n}$, we will use $\widetilde{S}$ to signify its realification to $\vecSp[H]^n \times \vecSp[H]^n$ through the Toeplitz decomposition, e.g., the realification of $\pbcone{[\alpha,\beta]}$ is

\begin{align*}
    \pbcone*{[\alpha,\beta]} := \left\{\begin{bNiceMatrix}
        H \\ S
    \end{bNiceMatrix} \in \vecSp[H]^n \times \vecSp[H]^n \mid 
        H+\ii* S \in \pbcone{[\alpha,\beta]}
    \right\}.
\end{align*}  

Next, given $\alpha,\beta \in \rF$, let us define a linear map
\begin{align} \label{eqa:mapr}
    \mapr^n := \begin{bNiceMatrix}
        - \sin \alpha & \cos \alpha \\
        \sin \beta & - \cos \beta 
    \end{bNiceMatrix} \otimes \mati_n
\end{align} where $\mati_n$ is the identity matrix of size $n$ and $\otimes$ denotes the Kronecker product. The superscript $n$ which we use to indicate the matrix size will be omitted if no ambiguity arises.
Then the phase-bounded cone $\pbcone{[\alpha,\beta]}$ can be characterized as follows.
\begin{lemma} \label{lem:pbcone-cond} 
    Let $0\,<\,\beta - \alpha \leq \pi$. A matrix $C \in \matSp{n}{n}$ is in $\pbcone{[\alpha,\beta]}$ if and only if
        \begin{align} \label{pb-cond}
            \mapr\, \begin{bNiceMatrix}
                \hermp(C) \\ \shermp(C)
            \end{bNiceMatrix} \in \psdcone \times \psdcone.
        \end{align}
\end{lemma}
\begin{proof}
    Denote by $\cone{[\alpha,\beta]}\,\subseteq \mathbb{C}$ the convex cone substented by by angles $\alpha,\beta$. 
    From our previous discussion involving \cref{fig:nr-union-cone}, a matrix $C$ is in $\pbcone{[\alpha,\beta]}$ if and only if $W(C)\subseteq \cone{[\alpha,\beta]}$, where the latter translates, via the definition of numerical range, into a pointwise condition:
        for all unit vectors $x\in \mathbb{C}^n$, the corresponding point $x\ct* Cx$ falls within $\cone{[\alpha,\beta]}$.
    The membership of a point in $\cone{[\alpha,\beta]}$ can be checked via two rotation tests: after rotating $\cone{[\alpha,\beta]}$ and the point together such that the extreme rays defined by $\alpha$ and $\beta$ align with the nonnegative real axis, the rotated point should be above and below the real axis, respectively.
    These tests mathematically translate into
    \begin{align*}
        \tIm(e^{-\ii* \alpha} x\ct*Cx) &= x\ct*\left(-\sin\alpha\,\hermp(C) + \cos\alpha\,\shermp(C) \right) x\geq 0, \\
        \tIm(e^{-\ii* \beta}  x\ct*Cx) &= x\ct*\left(-\sin\beta\,\hermp(C) + \cos\beta\,\shermp(C) \right) x\leq 0
    \end{align*}    
    for all $x\ct* x = 1$. In other words, $C \in \pbcone{[\alpha,\beta]}$ if and only if $-\sin\alpha\,\hermp(C) + \cos\alpha\,\shermp(C)\succcurlyeq\mato$ and $\sin\beta\,\hermp(C) -\cos\beta\,\shermp(C) \succcurlyeq \mato$, which can be written in the compact form \cref{pb-cond}.
\end{proof}

\begin{remark} \label{rmk:pb-cond}
    As we have discussed after \cref{eqa:pb-cone}, when $\beta-\alpha = 0$, $\pbcone{[\alpha,\beta]}$ is an essentially PSD cone.
    In such case, condition \cref{pb-cond} reduces to an equality constraint requiring that $\shermp((e^{-\ii*\alpha} C)) = \mato$ and we can instead consider $e^{-\ii* \alpha}\pbcone{[\alpha,\beta]}$ in the space $\vecSp[H]^n$. 
    The completion and the decomposition problems considered in the next few sections then reduce to the classical PSD completion and decomposition problems. 
    In the other extreme case, where $\beta-\alpha = \pi$, $\pbcone{[\alpha,\beta]}$ is an essentially accretive cone. Note that the two matrix inequalities encapsulated in condition \cref{pb-cond} now coincide, requiring that $\shermp((e^{-\ii*\alpha}C)) \succcurlyeq \mato$ while leaving $\hermp((e^{-\ii*\alpha}C))$ completely free.
    Therefore, the completion and decomposition problems in this case are again equivalent to a classical PSD completion and decomposition problems.
\end{remark}
Henceforth, we will always assume that $0<\beta-\alpha<\pi$. In this case, \cref{lem:pbcone-cond} builds up a bijection between the realified phase-bounded cone $\pbcone*{[\alpha,\beta]}$ and $\psdcone\times \psdcone$.
We summarize our understanding obtained so far in the first three rows of \cref{tab:pb-char}. 
Next, we shall define inner products on these ambient spaces (as indicated in the fourth row) so as to establish the duality between the phase-bounded completion and decomposition problems.

\begin{table}[h]
    \centering
    \renewcommand{\arraystretch}{1.5} 
    \begin{tabular}{@{}p{8em}@{}c@{}c@{}ccc@{}}
        \hline
        Ambient space&    $\mathbb{C}^{n\times n}$ & & $\vecSp[H]^n \times \vecSp[H]^n$ & & $\vecSp[H]^n \times \vecSp[H]^n$   \\
        Element & $C$ & $\xleftrightarrow[\text{decomposition}]{\text{Toeplitz}}$ & $\begin{bNiceArray}{c}
            \hermp(C) \\ \shermp(C)
        \end{bNiceArray}$ & $\xleftrightarrows[\mapr]{\mapr\inv}$ & $\begin{bNiceArray}{c} -\sin\alpha\,\hermp(C) + \cos\alpha\,\shermp(C) \\ \sin\beta\,\hermp(C) -\cos\beta\,\shermp(C) \end{bNiceArray}$ \\
        Phase-bounded set & $\pbcone{[\alpha,\beta]}$ & & $\pbcone*{[\alpha,\beta]}$ & & $\psdcone \times \psdcone$  \\ \hline
        Inner product & $\iprod(\cdot)(\cdot)[\rF]$ & & $\iprod(\cdot)(\cdot)[{\alpha,\beta}]$ & & $\iprod(\cdot)(\cdot)[\vecSp[H]]$ \\ \hline
    \end{tabular}
    \caption{Realification of $\mathbb{C}^{n\times n}$ and structure endowment via inner products.}  \label{tab:pb-char}
\end{table}

\section{Completability, decomposability, and their duality} \label{sec:comp-dcomp-cond}
\vphantom{\,}
Considering $\mathbb{C}^{n\times n}$ as the usual complex vector space, one commonly used structure is endowed by the trace inner product:   $\iprod(A)(B)[\mathrm{Tr}] = \Tr(A\ct* B) \in \mathbb{C}$. While realifying $\mathbb{C}^{n\times n}$ to a real space, we may induce the corresponding real inner product from the trace inner product. 
For example, we can simply treat $\mathbb{C}^{n\times n}$ itself as a real vector space (by restricting scalar multiplication to the reals $\rF$ only), an immediate inner product can be obtained by taking the real part of the trace inner product: $\iprod(A)(B)[\rF]:= \tRe(\iprod(A)(B)[\mathrm{Tr}])$. 
Additionally, note that the trace inner product is also well defined over $\rF^{n\times n} \times \rF^{n\times n}$ or $\vecSp[H]^n \times \vecSp[H]^n$. Nonetheless, one can verify via some algebraic manipulation that
\[
    \iprod(A)(B)[\rF] = 
    \iprod(\begin{bNiceMatrix}\tRe(A)\\\tIm(A)\end{bNiceMatrix})(\begin{bNiceMatrix}\tRe(B)\\\tIm(B)\end{bNiceMatrix})[\mathrm{Tr}] 
    = \iprod(\begin{bNiceMatrix}\hermp(A)\\\shermp(A)\end{bNiceMatrix})(\begin{bNiceMatrix}\hermp(B) \\ \shermp(B)\end{bNiceMatrix})[\mathrm{Tr}].
\]
For this equivalency and the same reason as before, we shall proceed with the realification to $\vecSp[H]^n \times \vecSp[H]^n$ with $\iprod(\cdot)(\cdot)[\vecSp[H]]:= \iprod(\cdot)(\cdot)[\mathrm{Tr}]$. It is well known that $\psdcone$ is self-dual in $\vecSp[H]^n$ with respect to the trace inner product (see \cite[Chp.~13]{vandenbergheChordalGraphsSemidefinite2015} for example). Hence, it follows from the expanded form of $\iprod(\cdot)(\cdot)[\vecSp[H]]$ and \cref{clm:dual-product} that $\psdcone\times \psdcone$ is self-dual in $\vecSp[H]^n \times \vecSp[H]^n$ with respect to $\iprod(\cdot)(\cdot)[\vecSp[H]]$.

Given $\alpha,\beta$ such that $0 < \beta - \alpha < \pi$, define the following binary operation 
\begin{align*}
    \iprod(\begin{bNiceMatrix}
        H_1 \\
        S_1
    \end{bNiceMatrix})(\begin{bNiceMatrix}
        H_2 \\
        S_2
    \end{bNiceMatrix})[\alpha,\beta] := \iprod(\mapr\begin{bNiceMatrix}
        H_1 \\
        S_1
    \end{bNiceMatrix})(\mapr\begin{bNiceMatrix}
        H_2 \\
        S_2
    \end{bNiceMatrix})[\vecSp[H]]\in\rF 
\end{align*}
for all $(H_1,S_1), (H_2, S_2)\in \vecSp[H]^n \times \vecSp[H]^n$. 
By noting that $\mapr$ is invertible when $0<\beta-\alpha<\pi$, we have the following lemma, which is a direct consequence of \cref{clm:self-dualization} and \cref{lem:pbcone-cond}. 
Unless otherwise specified, the ambient realified space is always set to be \emph{the space $\vecSp[H]^n \times \vecSp[H]^n$ equipped with $\iprod(\cdot)(\cdot)[\alpha,\beta]$} hereafter.
\begin{lemma} \label{lem:well-def-ip}
    For given $\alpha,\beta$ where $0<\beta-\alpha<\pi$, the phase-bounded cone $\pbcone*{[\alpha,\beta]}$ is self-dual in the Hilbert space $(\vecSp[H]^n \times \vecSp[H]^n,\iprod(\cdot)(\cdot)[{\alpha,\beta}])$.
\end{lemma}

In this context, it is crucial to recognize that the Toeplitz decomposition, the linear map $\mapr$ and its inverse all preserve the sparsity pattern of matrices. In other words, for any undirected graph $\mathcal{G}$, the subspace $\hmatofgraph*(\mathcal{G}) \times \hmatofgraph*(\mathcal{G})$ yields the realified $\cmatofgraph*(\mathcal{G})$, whereas it is also an invariant subspace of the invertible map $\mapr$ and its inverse.
        
We are now ready to define the following four important sets in $\mathbb{C}^{n\times n}$, which encompass the phase-bounded completion and decomposition problems.
\begin{align}
    \pbcone{[\alpha,\beta]}(\mathcal{G}) &:= 
        \cmatofgraph*(\mathcal{G}) \,\cap\, \pbcone{[\alpha,\beta]},   
        \label{cone-pb} \\
    \ppbcone{[\alpha,\beta]}(\mathcal{G}) &:=
        \left\{
            C \in \mathbb{C}^{n\times n}\mid \submat{C}(\mathcal{K})
                \in \pbcone{[\alpha,\beta]}
                \text{ for all cliques }\mathcal{K} 
                \text{ of } \mathcal{G}
        \right\}, 
        \label{cone-ppb} \\
    \compcone{[\alpha,\beta]} &:=
        \left\{
            C \in \mathbb{C}^{n\times n}\mid 
            \exists\, B \in \cmatofgraph*(\compg(\mathcal{G})) 
            \text{ such that }
            C + B \in \pbcone{[\alpha,\beta]}
        \right\}, 
        \label{cone-completable}\\
    \dcompcone{[\alpha,\beta]} &:= 
        \left\{
            C \in \cmatofgraph*(\mathcal{G})\mid  
                \exists\, n_C \in \rF[N](+) \text{ such that }
                    \textstyle  C = \sum_{i=1}^{n_C} C_i  \right.  \label{cone-decomposable} \\
                & \hspace{8em} \left. \vphantom{\textstyle \sum_{i=1}^{n_C}} \text{ where }\rank(C_i) = 1, C_i \in \pbcone{[\alpha,\beta]}(\mathcal{G}) \right\}. \nonumber
\end{align}
The set $\pbcone{[\alpha,\beta]}(\mathcal{G})$ collects phase-bounded matrices with sparsity pattern $\mathcal{G}$. Matrices in $\ppbcone{[\alpha,\beta]}(\mathcal{G})$ have each of its clique supporting matrices being phase-bounded.
The set $\compcone{[\alpha,\beta]}$ can be understood as the \emph{completable set} in the following sense: For any $C \in \compcone{[\alpha,\beta]}$, if we convert it into a partial matrix by letting all $C_{ij}$ where $\{i,j\}\notin \eset*$ free for completion, then the definition of $\compcone{[\alpha,\beta]}$ indicates that this phase-bounded completion problem is always solvable.
Meanwhile, $\dcompcone{[\alpha,\beta]}$ is the \emph{decomposable set} of matrices which admit a rank one decomposition with each summand being a clique supporting phase-bounded matrix. 
\begin{remark} \label{rmk:rank-one}
    Note that the rank-one representation of the sum, as in \cref{cone-decomposable}, follows the style of \cite{bakonyiMatrixCompletionsMoments2011}. This rank constraint can be easily removed. 
    When $\beta-\alpha<\pi$, we can deduce from \cref{rmk-def-1} that any matrix $C\in \pbcone{[\alpha,\beta]}$ of rank $r$ must admit the decomposition $C = T D T\ct* = \sum_{i=1}^r e^{\ii*\phi_i} t_i t_i\ct*$ where $T = [t_1\,\cdots\,t_r] \in \mathbb{C}^{n\times r}$ has full rank, $\phi_i \in [\alpha,\beta]$ for all $i$, and $D = \oplus_{i=1}^{r} e^{\ii*\phi_i}$.
    \begin{align} \label{eqa:clq-form-decompc}
        & \left\{\vphantom{\pbcone{[\alpha,\beta]}(\tilde{\mathcal{G}}(\mathcal{K}_i))}
            C \in \cmatofgraph*(\mathcal{G}) \mid \exists\, n_C \in \rF[N]_+ \text{ such that } \textstyle C = \sum_{i=1}^{n_C} C_i \right. \\ 
        & \hspace{8em} \left. \text{where } C_i \in \textstyle \pbcone{[\alpha,\beta]}(\tilde{\mathcal{G}}(\mathcal{K}_i))\text{ and }\mathcal{K}_i\text{ is some clique of }\mathcal{G}\right\}. \nonumber 
    \end{align}
\end{remark}
 
Through the Toeplitz decomposition, the four sets \cref{cone-pb}-\cref{cone-decomposable} induce their respective realified version: $\pbcone*{[\alpha,\beta]}(\mathcal{G}), \ppbcone*{[\alpha,\beta]}(\mathcal{G}), \compcone*{[\alpha,\beta]}$, and $\dcompcone*{[\alpha,\beta]} \subseteq \vecSp[H]^n \times \vecSp[H]^n$. To elucidate this, for example, $\compcone*{[\alpha,\beta]}(\mathcal{G})$ is just
\begin{align*}
    \left\{\begin{bNiceMatrix}
        H \\ S
    \end{bNiceMatrix} \in \vecSp[H]^n \times \vecSp[H]^n \mid 
        \exists 
        \begin{bNiceMatrix}
            X \\ Y
        \end{bNiceMatrix} \in \hmatofgraph*(\compg(\mathcal{G}))\times \hmatofgraph*(\compg(\mathcal{G})) \text{ such that } (H+\ii* S) + (X+\ii* Y) \in \pbcone{[\alpha,\beta]}
    \right\}.
\end{align*}
We summarize the key properties of these sets in the following theorem.
\begin{theorem} \label{thm:dual-cone}
    Given $\alpha,\beta$ where $ 0 < \beta - \alpha < \pi$, 
        and an undirected graph $\mathcal{G}$, the following statements hold:
    \begin{enumerate}
        \item Sets $\pbcone*{[\alpha,\beta]}(\mathcal{G}), \ppbcone*{[\alpha,\beta]}(\mathcal{G}), \compcone*{[\alpha,\beta]}$, and $\dcompcone*{[\alpha,\beta]}$ are closed pointed convex cones.
        \item The four cones have the following inclusion relationship:
            \begin{align*}
                \dcompcone*{[\alpha,\beta]}(\mathcal{G}) \subseteq
                \pbcone*{[\alpha,\beta]}(\mathcal{G}) \subseteq
                \compcone*{[\alpha,\beta]}(\mathcal{G}) \subseteq 
                \ppbcone*{[\alpha,\beta]}(\mathcal{G}).
            \end{align*}
        \item The four cones form two pairs of dual cones:
            \begin{align*}
                \dcone*(\pbcone*{[\alpha,\beta]}(\mathcal{G})) 
                    = \compcone*{[\alpha,\beta]}(\mathcal{G}),\quad 
                \dcone*(\ppbcone*{[\alpha,\beta]}(\mathcal{G})) 
                    = \dcompcone*{[\alpha,\beta]}.
            \end{align*}
    \end{enumerate}
\end{theorem}

\begin{proof}
    The proof relies on the well-established PSD results. Instead of introducing a new set of notations, note that when $\alpha = \beta =0$, sets \cref{eqa:pb-cone} and \cref{cone-pb} to \cref{cone-decomposable}, abbreviated as $\pbcone{[0]}, \pbcone{[0]}(\mathcal{G}),\ppbcone{[0]}(\mathcal{G}),\compcone{[0]}$ and $\dcompcone{[0]}$, are exactly their PSD counterparts. In this case, all these sets are already constrained in the real space $\vecSp[H]^n$ and thus we do not need to realify these sets. It is well known that \cref{thm:dual-cone} holds for these sets if we consider the ambient space as $(\vecSp[H]^n,\iprod(\cdot)(\cdot)[\vecSp[H]^n])$ \cite[Prop.~1.2.1]{bakonyiMatrixCompletionsMoments2011}. 

    Recalling the pattern-preserving property of $\mapr$, we now show that the sets in statement 1 can be mapped, by $\mapr$, to their PSD counterparts. To begin with, we have shown in \cref{lem:pbcone-cond} that $\mapr\, \pbcone*{[\alpha,\beta]} = \psdcone \times \psdcone = \pbcone{[0]} \times \pbcone{[0]}$.
    Note that $\pbcone*{[\alpha,\beta]}(\mathcal{G}) = \pbcone*{[\alpha,\beta]} \cap (\hmatofgraph*(\mathcal{G}) \times \hmatofgraph*(\mathcal{G}))$, then 
    \begin{alignat}{2}
        \mapr\, \pbcone*{[\alpha,\beta]}(\mathcal{G}) 
        &= \mapr\, ( \pbcone*{[\alpha,\beta]} \cap (\hmatofgraph*(\mathcal{G}) \times \hmatofgraph*(\mathcal{G})) ) && \label{eqa:pbcone-gph} \\
        &= \mapr\, \pbcone*{[\alpha,\beta]}  \cap \mapr\, (\hmatofgraph*(\mathcal{G}) \times \hmatofgraph*(\mathcal{G})) &&\nonumber\\
        &=  (\pbcone{[0]} \times \pbcone{[0]}) \cap  (\hmatofgraph*(\mathcal{G}) \times \hmatofgraph*(\mathcal{G})) 
            &&=  \pbcone{[0]}(\mathcal{G}) \times \pbcone{[0]}(\mathcal{G}) \nonumber
    \end{alignat} where the second equality holds since $R_{\alpha,\beta}$ is invertible.
    A graph $\mathcal{G}$ with a finite vertex set can only have finite number of maximal cliques, which we denote by $\mathcal{K}_1,\ldots,\mathcal{K}_m$.
    The partially phase-bounded set can be rewritten as the intersection of a finite collection of clique-defined sets, and it follows that
    \begin{align}
        \mapr\, \ppbcone*{[\alpha,\beta]}(\mathcal{G}) &= \mapr\, \bigcap_{i=1}^m (\pbcone*{[\alpha,\beta]}(\tilde{\mathcal{G}}(\mathcal{K}_i))+\hmatofgraph*(\compg(\tilde{\mathcal{G}}(\mathcal{K}_i)))\times \hmatofgraph*(\compg(\tilde{\mathcal{G}}(\mathcal{K}_i)))) \label{eqa:ppbcone-gph} \\
        &= \bigcap_{i=1}^m (\mapr\, \pbcone*{[\alpha,\beta]}(\tilde{\mathcal{G}}(\mathcal{K}_i)) + \mapr\, (\hmatofgraph*(\compg(\tilde{\mathcal{G}}(\mathcal{K}_i)))\times \hmatofgraph*(\compg(\tilde{\mathcal{G}}(\mathcal{K}_i))) )) \nonumber\\
        &= \bigcap_{i=1}^m (\pbcone{[0]}(\tilde{\mathcal{G}}(\mathcal{K}_i)) \times \pbcone{[0]}(\tilde{\mathcal{G}}(\mathcal{K}_i)) + \hmatofgraph*(\compg(\tilde{\mathcal{G}}(\mathcal{K}_i)))\times \hmatofgraph*(\compg(\tilde{\mathcal{G}}(\mathcal{K}_i))) ) \nonumber\\
        &= \ppbcone{[0]}(\mathcal{G}) \times \ppbcone{[0]}(\mathcal{G}). \nonumber
    \end{align}
    Moving on to the completable set, by its definition, it equals to $\pbcone*{[\alpha,\beta]} + \hmatofgraph*(\compg(\mathcal{G}))\times \hmatofgraph*(\compg(\mathcal{G}))$. Hence, it follows that
    \begin{alignat}{2}
        \mapr\,\compcone*{[\alpha,\beta]}(\mathcal{G}) 
        &= \mapr\,(\pbcone*{[\alpha,\beta]} + \hmatofgraph*(\compg(\mathcal{G}))\times \hmatofgraph*(\compg(\mathcal{G}))) && \label{eqa:comp-gph}\\
        &= \mapr\,\pbcone*{[\alpha,\beta]} + \mapr\,(\hmatofgraph*(\compg(\mathcal{G}))\times \hmatofgraph*(\compg(\mathcal{G}))) && \nonumber\\ 
        &= \pbcone{[0]} \times \pbcone{[0]} + \hmatofgraph*(\compg(\mathcal{G}))\times \hmatofgraph*(\compg(\mathcal{G}))
         &&= \compcone{[0]} \times \compcone{[0]}. \nonumber
    \end{alignat}
    Finally, for the decomposable set, \cref{rmk-def-1} implies that we can rewrite $\dcompcone*{[\alpha,\beta]}$ as a sum of phase-bounded matrix sets defined by maximal cliques, each obtained by adding $\pbcone*{[\alpha,\beta]}$ of compatible size to the corresponding maximal-clique supporting matrix of a tall zero matrix. Hence we have
    \begin{align}
        \mapr\,\dcompcone*{[\alpha,\beta]} &= \mapr\,\sum_{i=1}^{m}\pbcone*{[\alpha,\beta]}(\tilde{\mathcal{G}}(\mathcal{K}_i)) \label{eqa:decomp-gph}\\
        &= \sum_{i=1}^m \mapr\,\pbcone*{[\alpha,\beta]}(\tilde{\mathcal{G}}(\mathcal{K}_i)) = \sum_{i=1}^m (\pbcone{[0]}(\tilde{\mathcal{G}}(\mathcal{K}_i)) \times \pbcone{[0]}(\tilde{\mathcal{G}}(\mathcal{K}_i))) \nonumber \\
        &= \sum_{i=1}^m \pbcone{[0]}(\tilde{\mathcal{G}}(\mathcal{K}_i)) \times \sum_{i=1}^m \pbcone{[0]}(\tilde{\mathcal{G}}(\mathcal{K}_i)) = \dcompcone{[0]}\times\dcompcone{[0]}. \nonumber
    \end{align} 

    With relations \cref{eqa:pbcone-gph} to \cref{eqa:decomp-gph}, the proof of the theorem can be completed by appealing to the corresponding result for the PSD counterpart. 
    For the first statement, note that the convexity, conic property, and the closedness of $\pbcone{[0]}(\mathcal{G}),\ppbcone{[0]}(\mathcal{G}),\compcone{[0]}$ and $\dcompcone{[0]}$ are preserved under the Cartesian product. Furthermore, the linear map $\mapr\inv$ preserves the convexity and the conic property in general. 
    In this particular case, the closedness is also preserved by $\mapr\inv$ due to \cref{clm:self-dualization}. 
        Hence, $\pbcone*{[\alpha,\beta]}(\mathcal{G})$, $\ppbcone*{[\alpha,\beta]}(\mathcal{G})$, $\compcone*{[\alpha,\beta]}$ and $\dcompcone*{[\alpha,\beta]}$ are all closed convex cones.
    Finally, provided that $0<\beta-\alpha<\pi$, $C \in \pbcone{[\alpha,\beta]}\setminus \{0\}$ implies $-C \notin \pbcone{[\alpha,\beta]}$. Then it is straightforward to check that these realified cones are pointed. 
    
    The second statement follows easily from the inclusion relationship from the PSD case: $\dcompcone{[0]}(\mathcal{G}) \subseteq
    \pbcone{[0]}(\mathcal{G}) \subseteq
    \compcone{[0]}(\mathcal{G}) \subseteq 
    \ppbcone{[0]}(\mathcal{G})$.
    By invoking \cref{clm:dual-product} and \cref{clm:self-dualization}, the statement 3 also follows from its PSD counterpart. To illustrate this, we consider the completable cone for example:
    \begin{align*}
        \dcone*(\compcone*{[\alpha,\beta]}(\mathcal{G})) &= \dcone*(\mapr\inv\, (\compcone{[0]} \times \compcone{[0]})) = \mapr\inv \dcone*(\compcone{[0]} \times \compcone{[0]})\\
        &= \mapr\inv \dcone*(\compcone{[0]}) \times \dcone*(\compcone{[0]}) = \mapr\inv\, (\ppbcone{[0]}(\mathcal{G}) \times \ppbcone{[0]}(\mathcal{G})) = \ppbcone*{[\alpha,\beta]}(\mathcal{G}).
    \end{align*}
    The proof is now complete.
\end{proof}

In previous studies of completion with respect to 
the positive definite cone \cite{gronePositiveDefiniteCompletions1984},
the cone of Euclidean distance matrices \cite{bakonyiEuclidianDistanceMatrix1995},
and the numerical range completion property \cite{hadwinNumericalRangesMatrix2000}, 
a necessary condition for a completion to exists is, in all cases, that all specified principal submatrices exhibiting their according properties. However, the converse is not true in general, i.e., it might not be possible to complete a partial matrix  to a desirable full matrix even if all specified principal submatrices fulfill the desirable property. This is illustrated by the following example, involving both the positive semidefinite and the phase-bounded properties.

\begin{example} \label{exp:unique-comp}
    Let $\mathcal{G} = \graph{\vset*}{\eset*}$ be an undirected graph with $n = |\vset*|\geq 4$. 
    Suppose that $\mathcal{G}$ is a path graph (see \cref{fig:path-graph} for example). 
    For $\alpha,\beta$ such that $0<\beta-\alpha<\pi$, we choose an arbitrary $\gamma \in [\alpha,\beta]$.
    Now consider the following two partial matrices $H_\theta$ and $C_\gamma$  with pattern $\mathcal{G}$:
        \begin{align*}
            H_\theta &= \toep(?,\ldots,?,e^{-\ii*\theta}, \boxed{1}\,, e^{\ii* \theta}, ?,\ldots,?),\\
            C_\gamma &=  
                \toep(?,\ldots,?,1,\boxed{e^{\ii* \gamma}}\,,e^{\ii* 2\gamma},?,\ldots,?).
        \end{align*}
    Matrix $H_\theta$ admits a unique PSD completion $K_\theta$ while $C_\gamma$ admits a unique completion $F_\gamma$ in $\pbcone{[\alpha,\beta]}$:
    \begin{align*}
        K_\theta &= 
        \toep(e^{-\ii* (n-1)\theta},\ldots,e^{-\ii*\theta},\boxed{1}\,,e^{\ii*\theta},\ldots,e^{\ii*(n-1)\theta}), \\
        F_\gamma &= 
        \toep(e^{-\ii*(n-2)\gamma},\ldots,1,\boxed{e^{\ii*\gamma}}\,,e^{\ii*2\gamma},\ldots,e^{\ii* n \gamma}). 
    \end{align*}
    Now we assume $\mathcal{G}$ is a chordless cycle instead (see \cref{fig:chordless-cycle}). 
    Consider partial matrices:
        \begin{align*}
            \tilde{H}_\theta &= 
                \toep(\bar{h},?,\ldots,?,e^{-\ii*\theta},\boxed{1}\,,e^{\ii*\theta},?,\ldots,?,h),\\
            \tilde{C}_\gamma &= 
                \toep(y,?,\ldots,?,1,\boxed{e^{\ii*\gamma}}\,,e^{\ii*2\gamma},?,\ldots,?,x).
        \end{align*}
    As a consequence of the aforementioned path graph case, it can be deduced that $\tilde{H}_\theta$ can be completed to a PSD matrix if and only if $h=e^{\ii*(n-1)\theta}$ while $\tilde{C}_\gamma$ can be completed to a matrix in $\pbcone{[\alpha,\beta]}$ if and only if $x=e^{\ii* n\gamma} \text{ and } y = e^{-\ii* (n-2)\gamma}$. 
    In other words, if we pick $h$ from $\left\{e^{\ii* \phi}\mid e^{\ii*\phi}\neq e^{\ii*(n-1)\theta}\right\}$, the corresponding $\tilde{H}_\theta$ is not PSD completable despite that each of its specified principal submatrices is PSD. Similarly, if we choose the values of $x$ and $y$ such that
        $\toep(y,\boxed{e^{\ii*\gamma}},x) \in \pbcone{[\alpha,\beta]}$
    while $x\neq e^{\ii* n\gamma}$ and $y\neq e^{-\ii*(n-2)\gamma}$ (one can verify that such pair of $x,y$ does exist and a trivial example is $x=y=0$), 
    the resulting partial matrix $\tilde{C}_\gamma$ is not completable to $\pbcone{[\alpha,\beta]}$, though each of its specified principal submatrices falls inside $\pbcone{[\alpha,\beta]}$.
\end{example}

\begin{figure}
    \centering
    \subfloat[sparsity pattern of $C_\gamma$: a path graph]{
        \begin{minipage}[c]{.47\textwidth}
            \centering
            \includegraphics[width=.8\textwidth,trim={0.1in 0.2in 0.1in 0.2in},clip]{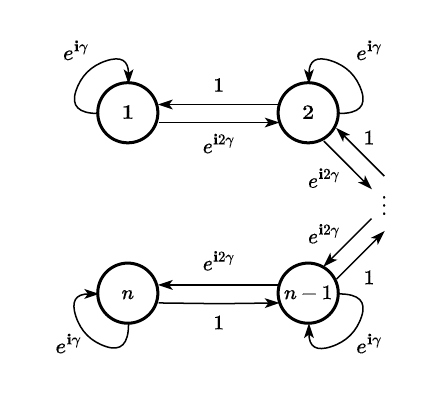} 
        \end{minipage}
        \label{fig:path-graph}
    }
    \subfloat[sparsity pattern of $\tilde{C}_\gamma$: a chordless cycle]{
        \begin{minipage}[c]{.47\textwidth}
            \centering        
            \includegraphics[width=.8\textwidth,trim={0.1in 0.2in 0.1in 0.2in},clip]{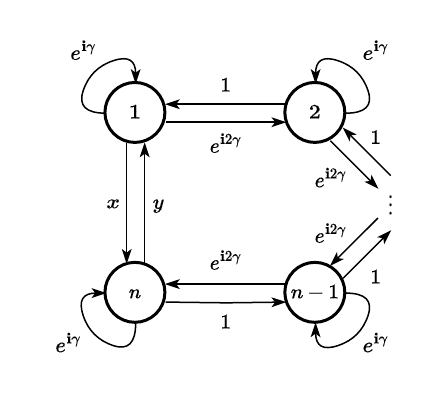} 
        \end{minipage}
        \label{fig:chordless-cycle}
    }
    \caption{\Cref{exp:unique-comp} shows that $\ppbcone{[\alpha,\beta]}(\mathcal{G})$ properly contains $\compcone{[\alpha,\beta]}(\mathcal{G})$ when the underlying graph $\mathcal{G}$ is not chordal}
\end{figure}

The above counterexamples illustrate that extra structural constraint on $\mathcal{G}$ is needed
in order for the necessary condition, i.e., that there is a completion with a specific property only if all specified principal submatrices have that property, to also be a sufficient condition.
This structural constraint turns out to be chordality for all abovesaid existing results in the literature. 
In the following theorem, we show that this is also the case for the phase-bounded property.
\begin{theorem} \label{thm:chordal-equiv}
    Given $\alpha,\beta \in \rF$ where $ 0 < \beta - \alpha < \pi$,
        and an undirected graph $\mathcal{G}$, the following statements are equivalent:
    \begin{enumerate}
        \item Graph $\mathcal{G}$ is chordal.
        \item $\dcompcone*{[\alpha,\beta]}(\mathcal{G}) = \pbcone*{[\alpha,\beta]}(\mathcal{G})$.
        \item $\compcone*{[\alpha,\beta]}(\mathcal{G}) = \ppbcone*{[\alpha,\beta]}(\mathcal{G})$.
    \end{enumerate}
\end{theorem}
\begin{proof}
    Recall the relations \cref{eqa:pbcone-gph} to \cref{eqa:decomp-gph} developed in the proof of \cref{thm:dual-cone}.
    Obviously, equality $2 \Leftrightarrow \dcompcone{[0]} = \pbcone{[0]}(\mathcal{G})$, and equality $3 \Leftrightarrow \compcone{[0]} = \ppbcone{[0]}(\mathcal{G})$.
    Then the theorem is a direct consequence of its PSD counterpart (see \cite[Thm.~2.4]{paulsenSchurProductsMatrix1989}, \cite[{Thm.~2.3}]{aglerPositiveSemidefiniteMatrices1988}):
    $\mathcal{G}$ is chordal $\Leftrightarrow \dcompcone{[0]} = \pbcone{[0]}(\mathcal{G}) \Leftrightarrow \compcone{[0]} = \ppbcone{[0]}(\mathcal{G})$.
\end{proof}

Through a complexification, i.e., converting $\left[\begin{smallmatrix}H \\ S\end{smallmatrix}\right] \in \vecSp[H]^n \times \vecSp[H]^n$ to $H+\ii* S \in \mathbb{C}^{n\times n}$, 
\Cref{thm:chordal-equiv} can be interpreted as the following corollary, which characterizes the completability and decomposability in phase-bounded cones.
\begin{corollary} \label{cor:comp-decomp}
    Let $\mathcal{G}$ be a chordal graph. Given $\alpha,\beta \in \rF$ where $0<\beta-\alpha<\pi$, the following statements hold:
    \begin{enumerate}
        \item A partial matrix $C$ with pattern $\mathcal{G}$ can be completed to a semisectorial matrix with phases bounded by $\alpha$ and $\beta$ if and only if each of its specified principal submatrices is semisectorial and has phases bounded by $\alpha$ and $\beta$.
        \item A full matrix $C$ with sparsity pattern $\mathcal{G}$ can be decomposed to a sum of rank one semisectorial matrices with phases bounded by $\alpha$ and $\beta$ if and only if $C$ itself is semisectorial and has phases bounded by $\alpha$ and $\beta$. 
    \end{enumerate}
\end{corollary}
\begin{remark} \label{rmk:banded}
    In particular, note that the so-called banded pattern \cite[{Sec.~2.6.1}]{bakonyiMatrixCompletionsMoments2011} is a special case of chordal pattern, 
        and \cite[{Thm.~9.1 and 10.1}]{wangPhasesComplexMatrix2020} is an immediate consequence of \cref{cor:comp-decomp}.
        We may paraphrase the completability condition in Corollary~\ref{cor:comp-decomp} as follows: a partial matrix $C$ with chordal pattern $\mathcal{G}$ can be completed to a semisectorial matrix if and only if all of its specified principal submatrices are semisectorial and their phases can be accommodated in a closed interval of length strictly less than $\pi$.
    When this condition is met, the smallest phase-bounded cone that $C$ can be completed to is $\pbcone{[\optimal(\alpha),\optimal(\beta)]}$ with
    \begin{align*}
        \optimal(\alpha) = \min_{\mathcal{K} \text{ is a clique of }\mathcal{G}}\, \minph(\submat{C}(\mathcal{K})),\quad 
        \optimal(\beta) = \max_{\mathcal{K} \text{ is a clique of }\mathcal{G}}\, \maxph(\submat{C}(\mathcal{K}))
    \end{align*}
    where the phases of all clique supporting matrices are chosen so that they fall within a $\pi$-interval.
\end{remark}

\begin{remark} \label{rem:further-generalization}
    It is also worth noting that not only the phase-related map $\mapr$, but also any invertible sparsity-pattern-preserving linear map $T$ on $\vecSp[H]^n \times \vecSp[H]^n$ can fit into the proof framework that leads to \cref{thm:dual-cone}, \cref{thm:chordal-equiv}, and subsequently, \cref{cor:comp-decomp}.
    As a consequence, similar results will hold for a matrix property, i.e., some matrix cone in $\mathbb{C}^{n\times n}$, that realifies to $T\inv\,(\pbcone{[0]}\times \pbcone{[0]})$.
\end{remark}

Finally, we conclude this section with an extension of the completability characterization in terms of strictly phase-bounded matrices $\spbcone{(\alpha,\beta)}$, which is the cone of \emph{sectorial matrices} with phases in the open interval $(\alpha,\beta)$. In this scenario, we may consider a directed graph (digraph) whose reflexive part is undirected. A digraph $\mathcal{G} = (\vset*,\eset*)$ is simply a graph which may have uni-directional edges, i.e., the edge set $\eset*$ might be asymmetric. Its reflexive part refers to the subgraph induced by the set $\vset*[R]$ of vertices with self-loops.
\begin{corollary} \label{cor:comp-ext}
    Let $\mathcal{G}$ be a digraph with undirected and chordal reflexive part $\mathcal{G}(\vset*[R])$.
    Provided $\alpha,\beta$ where $0<\beta-\alpha<\pi$, a partial matrix $C$ with pattern $\mathcal{G}$ can be completed to a matrix in $\spbcone{(\alpha,\beta)}$ if and only if $\submat{C}(\mathcal{K})$ is in $\spbcone{(\alpha,\beta)}$ for all cliques $\mathcal{K}\subseteq\, \vset*[R]$. 
\end{corollary}
\begin{proof}
    The ``only if'' part is straightforward, we show the ``if'' part.
    Assume that $\mathcal{G}$ has $n$ vertices, of which $k$ vertices are void of self-loop.
    We can always find a permutation matrix $P$ to put $C$ in the following form:
        \begin{align*} 
            P C P\tp = 
            \begin{bNiceArray}{cccc|c}[
                margin, first-row
            ]
                \Block{1-4}{k}& & & & n-k\\
                ?   & \star     & \cdots &  \star    &   \star  \\
                \star   & ? &   \ddots & \vdots          & \vdots                   \\
                \vdots  &  \ddots  & \ddots         & \star      & \vdots                       \\ 
                \star   & \cdots  & \star           & ?          & \star \\ \hline
                \star & \cdots & \cdots & \star & \submat{C}(\vset*[R])
            \end{bNiceArray} \quad
            \left(\begin{NiceArray}{c}
                \text{entries marked by } \star \text{ can be either} \\
                        \text{specified or ``?'', which depends on }\mathcal{G}
            \end{NiceArray}\quad\right).
        \end{align*}
    The submatrix $\submat{C}(\vset*[R])$ is a partial matrix with pattern $\mathcal{G}(\vset*[R])$.
    It is not hard to verify that the strictly phase-bounded version of \cref{cor:comp-decomp} is also true, i.e., we can first complete the partial submatrix $\submat{C}(\vset*[R])$ to a full matrix $\submat{K}(\vset*[R])$ in $\spbcone{(\alpha,\beta)}$.
    By replacing $\submat{C}(\vset*[R])$ in $PCP\tp$ with $\submat{K}(\vset*[R])$, we obtain another partial matrix $\hat{C}$.
    Denote the first $k$ diagonal entries of $\hat{C}$ by $d_1,\ldots,d_k$, which are to be determined. We have:
    \setlength{\arraycolsep}{.8pt}
    \NiceMatrixOptions{cell-space-limits = 2pt}
    \begin{align*}
        \mapr^n \begin{bNiceMatrix}
            \hermp(\hat{C}) \\
            \shermp(\hat{C})
        \end{bNiceMatrix} = 
            \begin{bNiceArray}{ccccc}[margin]
                \left|d_1\right| \sin(\angle d_1 - \alpha)    & \star          & \cdots  & \star & \star \\
                \star & \left|d_2\right| \sin(\angle d_2 - \alpha)    & \ddots  & \vdots  &\vdots \\
                \vdots & \ddots   &\ddots & \star & \vdots      \\
                \star & \cdots & \star & \left|d_k\right| \sin({\angle d_k - \alpha})    & \star  \\ 
                \star & \cdots & \star & \star & \submat{H}(\vset*[R]) \\ \hline
            \left|d_1\right| \sin(\beta - \angle d_1)    & \star          & \cdots  & \star & \star \\
                \star & \left|d_2\right| \sin(\beta - \angle d_2)    & \ddots  & \vdots  &\vdots \\
                \vdots & \ddots   &\ddots & \star & \vdots      \\
                \star & \cdots & \star & \left|d_k\right| \sin({\beta - \angle d_k})    & \star  \\ 
                \star & \cdots & \star & \star & \submat{S}(\vset*[R])
            \end{bNiceArray}
        \end{align*}
    where $\left[\submat{H}(\vset*[R])\hspace{.4em} \submat{S}(\vset*[R]) \right]\ct* = \mapr^{n-k}\left[\hermp((\submat{K}(\vset*[R]))) \hspace{.4em} \shermp((\submat{K}(\vset*[R]))) \right]\ct*$. 
    We can show, following the same line of proof for \cref{lem:pbcone-cond}, that $\mapr\,\spbcone*{(\alpha,\beta)}  = \pdcone \times \pdcone$.
    Since $\submat{K}(\vset*[R]) \in \spbcone{(\alpha,\beta)}$, then $\submat{H}(\vset*[R])$ and $\submat{S}(\vset*[R])$ are both positive definite.
    Obviously, for each $i=1,\ldots,k$, we can pick $\angle d_i$ from the open interval $(\alpha,\beta)$ and then set $|d_i|$ to be sufficiently large such that $\mapr^n [\hermp(\hat{C})\hspace{.5em}\shermp(\hat{C})]\ct*$ falls within $\pdcone \times \pdcone$ regardless of the remaining entries. 
    This means that $\hat{C}$ can always be completed within $\spbcone{(\alpha,\beta)}$, and thus, so can the original partial matrix $C$. This completes the proof. 
\end{proof}

\section{Characterization of all phased-bounded completions} \label{sec:characterization}
In this section, we turn to the synthesis side of the phase-bounded matrix completion problem. 
Given $\alpha,\beta$ where $0<\beta-\alpha<\pi$, and a (generalized) banded graph $\mathcal{G} = \graph{\vset*}{\eset*}$ as defined in \cite{bakonyiCentralMethodPositive1992},  we consider a partial matrix $C$ with pattern $\mathcal{G}$ and with all of its specified principal submatrices falling within $\pbcone{[\alpha,\beta]}$.
Note that the graph $\mathcal{G}$ is said to be \emph{banded} if $\{i_0,j_0\}\in \eset*$ ($i_0<j_0$) implies $\{i,j\} \in \eset*$ for all $i, j$ satisfying $i_0\leq i < j \leq j_0$.
A banded graph can be verified, through its definition and \cref{lem:chordal-peo}, to be chordal and is by default in a perfect elimination order. 
Then by \cref{cor:comp-decomp}, $C$ must admit at least one completion in $\pbcone{[\alpha,\beta]}$.
Our goal is to parameterize the following set of all phase-bounded completions of $C$:
\begin{align}
    \compset*(C)[\alpha,\beta] :=  \left\{
        K \in \pbcone{[\alpha,\beta]}: K-\cenComp[K](C)\in \cmatofgraph*(\compg(\mathcal{G}))
    \right\} \label{set:all-comp}
\end{align} where $\cenComp[K](C)$ can be any completion of $C$.
\begin{remark} \label{rmk:banded-no-chordal}
    We focus on the banded pattern here for two reasons: (1) in this case, the parameterization for all completions admits a neat closed form; and (2) in the more general chordal case, we can iteratively complete the matrix following the perfect elimination ordering, where the completion problem in each iterative step involves only a banded pattern.
\end{remark}

We know, from \cref{lem:pbcone-cond} and the pattern-preserving property of $\mapr$, that the phase-bounded completion problem can be equivalently converted to two independent PSD completion problems concerning partially PSD matrices $\hermp[\alpha](C)$ and $\shermp[\beta](C)$ where $[\hermp[\alpha](C)\hspace{.5em}\shermp[\beta](C)]\ct*\,:=\mapr [\hermp(C)\hspace{.5em}\shermp(C)]\ct*$.
Hence, we first revisit the relevant PSD completion results. 
For a partially PSD matrix $H$ with pattern $\mathcal{G}$, its central completion $\cenComp(H)$  is defined as the particular completion found, in a ``staircase'' fashion, by \cref{alg:psd-central}.
\begin{algorithm}[t]
    \caption{
        Finding the central PSD completion $\cenComp(H)$ of $H$
    } \label{alg:psd-central}
    \algrenewcommand\algorithmicensure{\textbf{output:\,}}
    \algrenewcommand\algorithmicrequire{\textbf{input:\,}}

    \begin{algorithmic}[1]
        \Require Partially PSD matrix $H$, and its pattern $\mathcal{G} = \graph{\vset*}{\eset*}$.
        \Ensure Central PSD completion $\cenComp(H)$. 
        \Statex
        \State $J(0) \leftarrow 0;\quad l_0 \leftarrow 1$;
        \For{$i \leftarrow 1 \text{ \textbf{to} } |\vset*|$}
            \State $J(i) \leftarrow  \max \{j \mid \{i,j\} \in \eset*\}$;
            \If{$J(i)-J(i-1)>0$ and $i \geq 2$}
                \State $x \leftarrow i;\quad y \leftarrow J(l_0);\quad z \leftarrow J(i);\quad l_0 \leftarrow i$; 
                \If{$y-x+1>0$}
                    \State  $\setX \leftarrow \{1,  \ldots, x-1\}$;\hspace{1em}
                            $\setY \leftarrow \{x,  \ldots, y\}$;\hspace{1em}
                            $\setZ \leftarrow \{y+1,\ldots, z\}$;
                    \State  
                            $T_{22}\leftarrow \submat{H}(\setY)$;\hspace{.5em}
                            $T_{12}\leftarrow \submat{H}(\setX,\setY)$;\hspace{.5em}
                            $T_{23}\leftarrow \submat{H}(\setY,\setZ)$;
                    \State  $T_{13}\leftarrow T_{12} T_{22}\pinv T_{23}$;
                    \State  $\submat{H}(\setX, \setZ)     
                            \leftarrow T_{13}$,
                            $\submat{H}({\setZ, \setX})    
                            \leftarrow T_{13}\ct*$;
                    \LComment{assign the RHS value to the designated block of $H$.}
                \Else
                    \State  $\setX \leftarrow \{1,  \ldots, x-1\}; \quad
                            \setZ \leftarrow \{y+1,\ldots, z\}$;
                    \State  $\submat{H}(\setX, \setZ)  \leftarrow \mato*; \quad
                            \submat{H}({\setZ, \setX}) \leftarrow \mato*$;
                \EndIf
            \EndIf
        \EndFor
        \State $\cenComp(H) \leftarrow H$;
    \end{algorithmic}
\end{algorithm} 
        
\begin{remark} \label{rmk:psd-central} 
    \Cref{alg:psd-central} is essentially decomposing the completion problem to $3$ by $3$ subproblems, and completing the big matrix in a ``staircase'' manner. 
    Due to the inheritance principle and uniqueness of central completion \mbox{\cite[{p.~647}]{bakonyiEuclidianDistanceMatrix1995}}, 
        $\cenComp(H)$ generated here equals to the one obtained by setting all Schur parameters of missing entries of $H$ to zero \cite[{Thm.~2.5.5}]{bakonyiMatrixCompletionsMoments2011}.
\end{remark}
A matrix $\Gamma$ is said to be \emph{contractive} if its spectral norm, a.k.a., its largest singular value, is no greater than 1. Based on the central completion generated by \cref{alg:psd-central}, one can have the following contractive parameterization of the set of all PSD completions of $H$:
\begin{lemma}{\cite[{Thm.~2.6.3}]{bakonyiMatrixCompletionsMoments2011}}\label{lem:psd-char}
    Given a partially PSD matrix $H$ with banded pattern $\mathcal{G}$. Let $\cenComp(H)$ be the central completion of $H$  and $\cenComp(L), \cenComp(R)$ be the lower and upper Cholesky factors of $\cenComp(H)$, i.e., $\cenComp(H) = \cenComp(L)\ct* \cenComp(L) = \cenComp(R)\ct* \cenComp(R)$ where $\cenComp(L), \cenComp(R)$ are lower and upper triangular, respectively. 
    Then there exists a unitary matrix $\cenComp(W)$ such that $\cenComp(R) = \cenComp(W) \cenComp(L)$. And there is a one-to-one correspondence between the following two sets:
        \begin{align*}
            \compset*(H)[+]     &:= \left\{K \in \psdcone \mid  K-\cenComp[K](H)\in \cmatofgraph*(\compg(\mathcal{G}))\right\}, \\
            \mathbf{\Gamma}_G  &:= \left\{\vphantom{\rSp*(\cenComp(L))^\perp} \Gamma \in \cmatofgraph*(\compg(\mathcal{G})) \mid
            \|\Gamma\|_2\leq 1, \Gamma \text{ is upper triangular and behaves as a zero}
            \right. \\
            &\hspace{6em} \left. 
            \text{operator when acting from $\rSp*(\cenComp(R))^\perp$ or onto $\rSp*(\cenComp(L))^\perp$}\right\}.
        \end{align*}
    The correspondence $f: \mathbf{\Gamma}_{\mathcal{G}} \mapsto \compset*(P)[+]$ has a closed form as follows:
        \begin{align} \label{eqa:psd-char}
            f(\Gamma;\cenComp(H)) = \cenComp(R)\ct* \invct*(\mati+\cenComp(W)\Gamma)(\mati - \Gamma\ct* \Gamma)\inv(\mati+\cenComp(W)\Gamma)\cenComp(R).
        \end{align}
\end{lemma}

By running \cref{alg:psd-central} on partial matrices $\hermp[\alpha](C)$ and $\shermp[\beta](C)$, we can obtain their central PSD completions $\cenComp*(\hermp[\alpha](C))$ and $\cenComp*(\hermp[\beta](C))$.
By combining these central completions and \cref{lem:psd-char}, we are now ready to provide a parameterization of $\compset*(C)[\alpha,\beta]$: 
\begin{theorem}\label{thm:pha-char}
    Provided $\alpha,\beta$ where $0<\beta-\alpha<\pi$, let $C$ be a partial matrix with chordal pattern $\mathcal{G}$ and assume that its arbitrary completion $\cenComp[K](C)$ falls within $\ppbcone{[\alpha,\beta]}(\mathcal{G})$. 
    Then $\compset*(C)[\alpha,\beta]$ is nonempty and there is a one-to-one correspondence between the set $\compset*(C)[\alpha,\beta]$ and the set $\mathbf{\Gamma}_{\mathcal{G}}\times \mathbf{\Gamma}_{\mathcal{G}}$. The correspondence $g: \mathbf{\Gamma}_{\mathcal{G}}\times \mathbf{\Gamma}_{\mathcal{G}} \mapsto \compset*(C)[\alpha,\beta]$ is as follows:
    \begin{align} \label{eqa:pb-char}
        g(\Gamma_1,\Gamma_2;\cenComp*(\hermp[\alpha](C)), \cenComp*(\hermp[\beta](C))) = \frac{1}{\sin(\beta-\alpha)}\left(
                \ee(\ii*\beta) f(\Gamma_1;\cenComp*(\hermp[\alpha](C))) + 
                \ee(\ii*\alpha)f(\Gamma_2;\cenComp*(\hermp[\beta](C)))
            \right).
    \end{align}
\end{theorem}

\section{Conclusion} \label{sec:concl}
In this paper, we study the classical matrix completion and decomposition problems in terms of phase-bounded property in particular. 
Previous results on this subject \cite[{Section 9, 10}]{wangPhasesComplexMatrix2020} for banded matrices are extended to a more general matrix structure, namely the chordal structure. 
The proofs are based on establishing a duality between cones of phases-bounded matrices, and this duality also reveal the connection between the completion and the decomposition problems.
We also provide a PSD-completion-based characterization of all phase-bounded completions of a partially phase-bounded matrix with a banded pattern.

There are several unanswered questions and directions for future work. 
In a phase-bounded completion problem, one possible definition for a central phase-bounded completion is $\frac{1}{\sin(\beta-\alpha)}(
    \ee(\ii*\beta) \cenComp*(\hermp[\alpha](C)) + 
    \ee(\ii*\alpha)\cenComp*(\hermp[\beta](C)))$. A more meaningful interpretation of this unique completion would be expected. For example, does it have a banded inverse as in the case of positive definite completion \cite[Thm.~2]{gronePositiveDefiniteCompletions1984}? Does it give a maximizer of some physically interpretable quantity? 
We would also like to see a more compact characterization of all completions in lieu of decomposing the problem into two PSD problems. 
Besides, the dual-cone results for PSD completion and decomposition have been fruitfully used in decomposing and decentralizing large scale SDP problems in networked control \cite{rantzerDistributedPerformanceAnalysis2010,zhengChordalSparsityControl2019,zhengChordalFactorwidthDecompositions2021a}. We also expect the results herein can be applied to tackle large-scale phase-bounded conic programming problems arising in phase-based analysis.

\section*{Acknowledgments}
We would like to thank the referees for their careful reading and for their constructive comments on this paper, especially their suggestions on effectively exploiting the connection between the phase-bounded cone and the positive semidefinite cone which have significantly streamlined our arguments.

\appendix
\section{Proof of \cref{clm:self-dualization}} \label{pf:self-dualization}
Since we are considering the finite-dimensional case, to conclude that $\vecSp[X]$ equipped with $\iprod(\cdot)(\cdot)[T\vecSp[X]]$ is a Hilbert space, it suffices to verify that $\iprod(\cdot)(\cdot)[T\vecSp[X]]$ is a well-defined inner product. This is indeed the case since $\iprod(\cdot)(\cdot)[T\vecSp[X]] = \iprod(T \cdot)(T \cdot)[\vecSp[Y]] = \iprod( \cdot)(T\ct T \cdot)[\vecSp[Y]]$, considering that $\iprod(\cdot)(\cdot)[\vecSp[Y]]$ is an inner product and that $T\ct T$ is positive definite since $T$ is invertible and linear.

To prove the second statement, consider a convergent (or Cauchy) sequence $\{x_i\}$ in $T\inv Y$, the sequence $\{Tx_i\}$ must be a Cauchy sequence in $Y$, due to the fact that  $\|Tx_i - Tx_j \|_{\vecSp[Y]} = \|x_i - x_j \|_{T\vecSp[X]}$ where the norms are induced by the inner products. Since $Y$ is closed, there exists some $y \in Y$ such that $\lim_{i\rightarrow \infty} Tx_i = y$. The boundedness of $T$ then implies $T \lim_{i\rightarrow \infty} x_i = y$, or equivalently, $\lim_{i\rightarrow \infty} x_i =  T\inv y \in T\inv Y$. Therefore, $T\inv Y$ is indeed closed as the choice for $\{x_i\}$ is arbitrary.

To prove the third statement, note that for any $x \in \vecSp[X]$ and $y \in \vecSp[Y]$, it holds that $\iprod(x)(T\inv y)[T\vecSp[X]] = \iprod(Tx)(y)[\vecSp[Y]]$.
Thereby, $x \in T\inv \Cone\dcone*$, or equivalently $Tx \in \Cone\dcone*$, implies that $\iprod(x)(T\inv k)[T\vecSp[X]] = \iprod(Tx)(k)[\vecSp[Y]] \geq 0$ for all $k \in \Cone$. This shows that $T\inv \Cone\dcone* \subseteq \dcone*(T\inv \Cone)$. Conversely, $x_0 \notin T\inv \Cone\dcone*$ means there exists $k\in \Cone$ such that $\iprod(x_0)(T\inv k)[T\vecSp[X]] = \iprod(Tx_0)(k)[\vecSp[Y]] <0$. Hence, $x_0 \notin \dcone*(T\inv \Cone)$ as well and $\dcone*(T\inv \Cone) = T\inv \Cone\dcone*$. The remaining self-duality statement follows trivially from this equality and the proof is now complete.
 
\bibliographystyle{siamplain}
    
\end{document}